\documentclass[11pt, reqno]{amsart}
\usepackage{amssymb, amsthm, amsmath, amsfonts,mathtools}
\usepackage{array, epsfig}
\usepackage{bbm}
\usepackage{MnSymbol}

\usepackage{hyperref}
\usepackage[numbers,square]{natbib}
\usepackage{tikz-cd}%commutatve diagram

\allowdisplaybreaks

\numberwithin{equation}{section}

  \evensidemargin
\oddsidemargin
\newcommand{\Tr}{\mathop{\mathrm{Tr}}\nolimits}
\newcommand{\He}{\mathop{\mathrm{He}}\nolimits}

\newcommand{\ii}{{\rm{i}}}
\newcommand{\fnorm}{{\mathfrak{N}}}

\newcommand{\bC}{\mathbb{C}}

\newcommand{\bT}{\mathbb{T}}
\newcommand{\bH}{\mathbb{H}}
\newcommand{\bD}{\mathbb{D}}

\newcommand{\cD}{\mathcal{D}}

\DeclareMathOperator*{\cotanh}{cotanh}

\def\dint{\textup{d}}

\newcommand{\E}{\mathbb E}

\newcommand{\R}{\mathbb{R}}
\newcommand{\N}{\mathbb{N}}

\newcommand{\C}{\mathbb{C}}
\newcommand{\Z}{\mathbb{Z}}

\renewcommand{\Re}{\operatorname{Re}}
\renewcommand{\Im}{\operatorname{Im}}

\newcommand{\eps}{\varepsilon}

\newcommand{\toweak}{\overset{w}{\underset{n\to\infty}\longrightarrow}}
\newcommand{\tovague}{\overset{v}{\underset{n\to\infty}\longrightarrow}}
\newcommand{\toweakd}{\overset{w}{\underset{d\to\infty}\longrightarrow}}

\newcommand{\ton}{\overset{}{\underset{n\to\infty}\longrightarrow}}

\newcommand{\bsl}{\backslash}

\newcommand{\dd}{{\rm d}}
\newcommand{\eee}{{\rm e}}

\theoremstyle{plain}
\newtheorem{theorem}{Theorem}[section]
\newtheorem{lemma}[theorem]{Lemma}
\newtheorem{corollary}[theorem]{Corollary}
\newtheorem{proposition}[theorem]{Proposition}

\theoremstyle{definition}

\theoremstyle{remark}
\newtheorem{remark}[theorem]{Remark}

\begin{document}

\author{Zakhar Kabluchko}
\address{Zakhar Kabluchko\\
Institut f\"ur Mathematische Stochastik\\
Universit\"at M\"unster\\
Orl\'eans-Ring 10\\
48149 M\"unster, Germany}
\email{zakhar.kabluchko@uni-muenster.de}

%%%%%%%%%%%%%%%%%%%%%%%%%%%%%%%%%%%%%%%%%%%%%%%%%%%%%%%%%%%%%%%%%%%%%%%%%%%%%%%%%%%%%%%%%%%%%%%%%%
%%%% All numerical simulations and figures are in "Trigonometric Laguerre Polynomials.nb"
%%%%Exception: zeroes of the Laguerre polynomials were simulated in Poly_IID_Coeff.nb
%%%%%%%%%%%%%%%%%%%%%%%%%%%%%%%%%%%%%%%%%%%%%%%%%%%%%%%%%%%%%%%%%%%%%%%%%%%%%%%%%%%%%%%%%%%%%%%%%%

\title[Unitary Hermite polynomials]{Lee--Yang zeroes of the Curie--Weiss ferromagnet, unitary Hermite polynomials, and the backward heat flow}

\keywords{Curie--Weiss model; unitary Hermite polynomials; Lee--Yang zeroes; finite free probability; free multiplicative convolution; free unitary normal distribution; saddle-point method}

\subjclass[2010]{Primary: 82B20, 30C15; Secondary: 30C10, 60B10,
 46L54, 30F99, 30E15}

\begin{abstract}
The backward heat flow on the real line started from the initial condition $z^n$ results in the classical $n$-th Hermite polynomial whose zeroes are distributed according to  the Wigner semicircle law in the large $n$ limit.  Similarly, the backward heat flow with the periodic initial condition $(\sin \frac \theta 2)^n$ leads to  trigonometric or unitary analogues of the Hermite polynomials. These polynomials are closely related to  the partition function of the Curie--Weiss model and appeared in the work of Mirabelli on finite free probability. We relate the $n$-th unitary Hermite polynomial to the expected characteristic polynomial of a unitary random matrix obtained by running a Brownian motion on the unitary group $U(n)$. We identify the global  distribution of  zeroes of the unitary Hermite polynomials as the free unitary normal distribution.  We also compute the asymptotics of these polynomials or, equivalently,  the free energy of the Curie--Weiss model in a complex external field.
We identify the global distribution of the Lee--Yang zeroes of this model.
Finally, we  show that the backward heat flow applied to a high-degree real-rooted polynomial (respectively, trigonometric polynomial) induces, on the level of the asymptotic distribution of its roots, a  free Brownian motion (respectively, free unitary Brownian motion).
\end{abstract}

\maketitle
%\tableofcontents

\section{Introduction}
\subsection{Hermite polynomials and their trigonometric analogues}

One possible way to define the classical (probabilist) Hermite polynomials $\He_0(z) = 1, \He_1(z) = z, \He_2(z)=z^2-1,\ldots$  is the formula
\begin{equation}\label{eq:hermite_poly_def}
\He_n(z)
=
\exp\left\{-\frac 1 2  \partial_z^2\right\} z^n
=
n! \sum_{m=0}^{[n/2]}  \frac{(-1)^m }{m!  2^m} \cdot \frac{z^{n-2m}}{(n-2m)!},
\qquad n\in \N_0,
\end{equation}
where $\partial_z$ denotes differentiation in $z$ and the exponential can be understood as an infinite series which terminates after finitely many non-zero summands.   More generally, we have
\begin{equation}\label{eq:eq:hermite_exp_on_z^n_with_sigma}
\exp\left\{-\frac{1}{2}\sigma^2 \partial_z^2\right\} z^n =  \sigma^n \He_n\left(\frac{z}{\sigma}\right), \qquad \sigma >0.
\end{equation}
This can be expressed by saying that the $n$-th Hermite polynomial arises when solving the \textit{backward} heat equation $\partial_t f(z; t) = -\frac 12 \partial^2_z f(z;t)$ on the real line with the initial condition $f(z;0) = z^n$.

We shall be interested in the trigonometric (or unitary) analogues of the Hermite polynomials. To introduce them, it will be convenient to adopt the following (somewhat unconventional) terminology. A trigonometric polynomial $T_n(\theta)$ of degree $n\in \N_0$ is an expression of the form
\begin{equation}\label{eq:algebraic_trigonometric}
T_n(\theta) = \frac{P_n(\eee^{\ii \theta})}{\eee^{\ii n \theta/2}},
\end{equation}
where $P_n(z)\in \C[z]$ is an algebraic polynomial of degree $n$ such that $P_n(0)\neq 0$.  For example, if $P_n(z) = (z-1)^n$, then $T_n(\theta) = (2\ii \sin \frac \theta 2)^n$.

In general,  $T_n(\theta)$ is a linear combination of the functions $\theta \mapsto \eee^{\ii (k-\frac n2) \theta}$, $k=0,\ldots, n$, with complex coefficients (such that the first and the last coefficient do not vanish). If $n=2d$ is even, then $T_n(\theta)$ can be represented as a linear combination of the functions $1, \sin \theta, \cos \theta,\ldots, \sin (d\theta), \cos (d\theta)$. This case corresponds to the usual definition of trigonometric polynomials. If $n=2d+1$ is odd, then $T_n(\theta)$ can be written as a linear combination of the functions $\sin (\ell \theta/2)$ and $\cos (\ell \theta/2)$ with $\ell = 1,3,5,\ldots, n$. This case is somewhat unconventional.
From the representation $P_n(z) = C \prod_{j=1}^n (z-z_j)$, where $z_1,\ldots, z_n\in \C\backslash\{0\}$ are the complex zeroes of $P_n$, one deduces the existence of a representation
$$
T_n(\theta) = C' \prod_{j=1}^n \sin \frac {\theta - \theta_j} 2
$$
where $\theta_1,\ldots,\theta_n\in \C$ are chosen to satisfy $\eee^{\ii \theta_j} = z_j\neq 0$.

We shall be interested in trigonometric polynomials with real roots only or, equivalently, in algebraic polynomials having all roots on the unit circle. If we  consider $(z-1)^n$ as the ``simplest'' algebraic polynomial of degree $n$ with this property, then the ``simplest'' real-rooted trigonometric polynomial of degree $n$ is $(\sin \frac \theta 2)^{n}$  (up to a constant factor).
%Indeed, both polynomials have a multiplicity $n$ root at $0$.
%The trigonometric analogue of the algebraic polynomial $z^n$ (which has a zero of multiplicity $n$ at $z=0$) is the trigonometric polynomial $(2\sin \frac \theta 2)^{n} = (-\ii)^n (\eee^{\ii \theta} - 1)^n\eee^{-\ii n \theta/2}$ (which has a zero of multiplicity $n$ at $\theta=0$).
To derive the trigonometric analogues of the Hermite polynomials $\He_n$, we look at the backward heat flow on $\R$ with the initial condition $(\sin \frac \theta 2)^{n}$. More precisely, we take some parameter $\sigma^2>0$, let $\partial_\theta$ be the differentiation operator in $\theta$ and consider, similarly to~\eqref{eq:eq:hermite_exp_on_z^n_with_sigma}, the expression
\begin{align}
T_n(\theta; \sigma^2) := \exp\left\{-\frac{1}{2}\sigma^2 \partial_\theta^2\right\} \left(\sin \frac \theta 2\right)^{n}
&=
(2\ii)^{-n} \exp\left\{-\frac{1}{2}\sigma^2 \partial_\theta^2\right\} \frac{(\eee^{\ii \theta} - 1)^n}{\eee^{\ii n \theta/2}}\notag\\
&=
(2\ii)^{-n} \exp\left\{-\frac{1}{2}\sigma^2 \partial_\theta^2\right\}\sum_{j=0}^n (-1)^{n-j} \binom nj \eee^{\ii \theta (j - \frac n2)}\notag\\
&=
(2\ii)^{-n} \sum_{j=0}^n (-1)^{n-j} \binom nj  \eee^{\frac 12 \sigma^2 \left(j-\frac n2\right)^2} \eee^{\ii \theta (j - \frac n2)}.\label{eq:exp_diff_sin_theta_power}
\end{align}
In the last line we used the identity  $\eee^{-\frac 12 \sigma^2 \partial_\theta^2} \eee^{\ii  c \theta} = \eee^{\frac 12 \sigma^2 c^2} \eee^{\ii c \theta}$. The algebraic polynomials corresponding to these trigonometric polynomials via~\eqref{eq:algebraic_trigonometric}  are given by
\begin{equation}\label{eq:hermite_poly_circ_def1}
H_n(z;\sigma^2)
=
\sum_{j=0}^n   (-1)^{n-j} \binom nj \exp\left\{ - \frac {\sigma^2 j (n-j)}{2}\right\} z^j,
\qquad
n\in \N,
\end{equation}
up to a multiplicative constant which was chosen to make $H_n(z;\sigma^2)$ monic.

If we view $t := \sigma^2$ as the time, then the trigonometric polynomial $T_n(\theta; t)$ satisfies the backward heat equation $\partial_t T_n(\theta; t) = -\frac 12 \partial_\theta^2 T_n(\theta; t)$, while $H_n(z;t)$ satisfies  the PDE
$$
\partial_t H_n(z; t) = - \frac 12  (z \partial_z) (n- z\partial_z) H_n(z;t),
\qquad
H_n(z; 0) = (z-1)^n,
$$
very similar to the heat-type PDE's that appeared in~\cite{tao_blog2} and~\cite[Section~2.3]{hall_ho}. Indeed, using the fact that $(z \partial_z) (n- z\partial_z)  z^j = j(n-j) \cdot  z^j$, we deduce
\begin{align*}
\exp\left\{-\frac t2 (z \partial_z) (n- z\partial_z)\right\} (z-1)^n
&=
\sum_{j=0}^n (-1)^{n-j} \binom nj \exp\left\{-\frac t2  (z \partial_z) (n- z\partial_z)\right\} z^j\\
&=
\sum_{j=0}^n (-1)^{n-j} \binom nj  \exp\left\{ - \frac {t}{2} j (n-j)\right\} z^j
= H_n(z;t).
\end{align*}

\subsection{Connection to finite free probability}
In the following, we shall refer to the polynomials $H_n(z; \sigma^2)$ defined by~\eqref{eq:hermite_poly_circ_def1} as \textit{unitary Hermite polynomials} with parameter $\sigma^2>0$. These polynomials appeared in the work of Mirabelli~\cite{mirabelli_diss} on  finite free probability, a theory developed by Marcus, Spielman, Srivastava~\cite{marcus_spielman_srivastava} and  Marcus~\cite{marcus}.
This theory studies the \textit{finite free additive convolution} $\boxplus_{n}$ and the \textit{finite free multiplicative convolution} $\boxtimes_{n}$ which are bilinear operations on the space of algebraic polynomials of degree at most $n$  defined~\cite{marcus,marcus_spielman_srivastava} as follows:
\begin{align}
\left(\sum_{i=0}^{n} \frac{\alpha_i}{i!} z^i \right)\boxplus_{n}\left(\sum_{j=0}^{n} \frac{\beta_j}{j!} z^j \right)
&=
\frac 1 {n!} \sum_{\ell=0}^{n}  \frac{z^\ell}{\ell!} \sum_{\substack{i,j\in \{0,\ldots, n\}\\ i+j=n+\ell}} \alpha_i \beta_j, \label{eq:finite_free_add_conv_def}\\
\left(\sum_{j=0}^{n} \alpha_j z^j \right)\boxtimes_{n}\left(\sum_{j=0}^{n} \beta_j z^j \right)
&=
\sum_{j=0}^{n} (-1)^{n-j}  \frac{\alpha_j \beta_j}{\binom {n}{j}} z^j. \label{eq:finite_free_mult_conv_def}
\end{align}

It has been shown in~\cite{marcus,marcus_spielman_srivastava,mirabelli_diss} that there is an analogue of the central limit theorem for these convolutions (for every fixed $n\in \N$). The classical Hermite polynomials play the role of the normal distribution for $\boxplus_n$; see Theorem~6.7 in~\cite{marcus} and Theorems~3.2, 3.5 in~\cite{mirabelli_diss}. Similarly, the unitary Hermite polynomials $H_n(z; \sigma^2/(n-1))$ play the role of the normal distribution for $\boxtimes_n$; see Theorems~3.16, 3.23, 3.32 in~\cite{mirabelli_diss}. For example, the analogue of the de Moivre-Laplace theorem for $\boxtimes_{n}$ is as follows. Fix some even number $n = 2d$ and consider degree $2d$ polynomials
$$
Q_N(z) := (z^2 -2 z \cos (\sigma/\sqrt N) +1)^d, \qquad N\in \N,
$$
having two zeroes at $\eee^{\ii \sigma/ \sqrt N}$ and $\eee^{-\ii \sigma/ \sqrt N}$, both of multiplicity $d$. Then, it can be shown that
$$
\lim_{N\to\infty} \underbrace{Q_N(z) \boxtimes_{2d} \ldots \boxtimes_{2d} Q_N(z)}_{N \text{ times}} = H_{2d}\left(z; \frac{\sigma^2}{2d-1}\right).
$$
The operations $\boxplus_n$ (respectively, $\boxtimes_n$) are known (in a suitable sense) to converge, as $n\to\infty$, to the classical free additive (respectively, multiplicative) convolutions $\boxplus$, respectively, $\boxtimes$. For information on (infinite) free probability we refer to~\cite{voiculescu_nica_dykema_book} and~\cite{nica_speicher_book}, for its finite counterpart to~\cite{marcus,marcus_spielman_srivastava,arizmendi_perales,arizmendi_garza_vargas_perales,mirabelli_diss}.

\subsection{Connection to the Curie--Weiss model}
It has been observed by Mirabelli~\cite[Section~3.2.5]{mirabelli_diss} that the unitary Hermite polynomials are closely related to the \textit{Curie--Weiss model} (or the Ising model on the complete graph), which is one of the simplest models of statistical mechanics.
The partition function of the Curie--Weiss model at inverse temperature $\beta>0$ and with external magnetic field $h\in \R$ is given by
\begin{equation}\label{eq:curie_weiss_part_funct_def}
Z_n(\beta, h) = \sum_{(\sigma_1,\ldots, \sigma_n) \in \{\pm 1\}^n} \eee^{\frac {\beta}{2n} \left(\sum_{k=1}^n \sigma_k\right)^2 + h \sum_{k=1}^n \sigma_k}.
\end{equation}
For every $j\in \{0,\ldots,n\}$  there exist $\binom nj$ configurations $(\sigma_1,\ldots, \sigma_n)$ in which the number of $+1$'s is $j$. Since for every such configuration we have $\sum_{k=1}^n \sigma_k = 2j - n$,  the above partition function can be written as
\begin{equation}\label{eq:curie_weiss_part_funct_reduction_hermite}
Z_n(\beta, h)
=
\sum_{j=0}^n \binom n j \eee^{\frac {\beta}{2n} (2j-n)^2 + h(2j - n)}
=
H_n \left(-\eee^{2h}; \frac {4\beta}{n}\right) \cdot \eee^{n(\frac \beta 2 - h)} \cdot (-1)^n.
\end{equation}
The fact that various quantities related to $Z_n(\beta, h)$  satisfy PDE's (the heat equation, the Burgers equation and a Hamilton--Jacobi equation) is known since Newman~\cite{newman_perc_theory}; see also~\cite{dominguez_mourrat,dascaliuc}.

The behavior of the Curie--Weiss model at \textit{real} parameters $\beta$ and $h$ is very well understood; see~\cite{ellis_newman,ellis_newman_stat_CW,ellis_newman_rosen} as well as the books by Ellis~\cite[Sections~IV.4, V.9]{ellis_book} and Friedli and Velenik~\cite[Chapter~2]{friedli_velenik_book}. For a recent approach using the theory of mod-$\phi$-convergence we refer to~\cite{meliot_nikeghbali_stat_mech}.

The behavior at \textit{complex} parameters $\beta$ and $h$, and in particular the location of the complex zeroes of $Z_n(\beta, h)$ is also of interest and has attracted attention in the theoretical physics literature~\cite{glasser_etal,Krasnytska,deger_flindt,deger_brange_flindt}; see also the recent paper~\cite{newman_motion_yee_yang}. These authors were motivated by the Lee--Yang program~\cite{lee_yangI,lee_yangII,fisher} which relates phase transitions to the complex zeroes of the partition function.  Shamis and Zeitouni~\cite{shamis_zeitouni} analyzed the partition function and its zeroes  at complex $\beta$ (with $h=0$) in a small neighborhood of the critical value $\beta = 1$, while the behavior outside this neighborhood remains largely unknown. The results of the present paper clarify the asymptotic behavior of $Z_n(\beta, h)$ at complex $h$ (with fixed real $\beta>0$) and, in particular, identify the global limiting distribution of the complex zeroes of $Z_n(\beta, h)$. Thus, we analyse the so-called Lee--Yang zeroes in contrast to the Fisher zeroes analyzed in~\cite{shamis_zeitouni}.
Partition function zeroes of the Ising model on Cayley trees and hierarchical lattices have been studied in~\cite{bleher_etal_I,bleher_etal_II,chio_etal}, where pointers to earlier literature and a discussion of conjectures on the Ising model on $\Z^d$ can be found.

\subsection{Summary of results}
The main results of the present paper, to be stated in Section~\ref{sec:main_results}, can be summarized as follows:
\begin{itemize}
\item[(a)] We prove that the empirical distribution of zeroes of $H_n(z; \sigma^2/n)$ converges weakly on the unit circle to the free unitary normal distribution $\fnorm_{\sigma^2}$, thereby identifying the limiting distribution of the Lee--Yang zeroes of the Curie--Weiss model; see Theorem~\ref{theo:weak_conv_zeroes_to_free_normal} and Corollary~\ref{cor:curie_weiss_zeroes}.
\item[(b)]  We compute the asymptotics of $H_n (z;\sigma^2/n)$ for complex $z$ with $|z|\neq 1$, thereby identifying the free energy of the Curie--Weiss model with complex external field; see Theorem~\ref{theo:hermite_unitary_asympt} and Corollary~\ref{cor:curie_weiss_free_energy}.
\item [(c)] It is well known~\cite[Section~6.1.2]{potters_bouchaud_book} that the expected characteristic polynomial of a Wigner random matrix of size $n$ coincides with the $n$-th classical Hermite polynomial. We prove a unitary analogue of this result. More precisely, let $(U_n(t))_{t\geq 0}$ be a Brownian motion on the unitary group $U(n)$ such that $U_n(0)=I_n$ is the identity matrix. In Theorem~\ref{theo:expected_char_poly_unitary_BM} we show that
    $$
    \E \det (x I_n - U_n(t))= \eee^{-nt/2} H_n(\eee^{t/2} x; t), \qquad t\geq 0.
    $$
    An extension to Brownian motion starting at an arbitrary  unitary matrix is given in Corollary~\ref{cor:unitary_BM}.
\item[(d)] Consider a high-degree polynomial (respectively, trigonometric polynomials) whose roots are real and have certain asymptotic distribution. We show that applying the backward heat flow to the polynomial is equivalent, on the level of the asymptotic distribution of the roots, to starting a free Brownian motion (respectively, a free unitary Brownian motion) from the initial distribution of roots; see Theorems~\ref{theo:heat_flow_zeroes_algebraic} and~\ref{theo:heat_flow_zeroes}.
\item [(e)]  We review the properties of the free unitary normal distribution $\fnorm_{\sigma^2}$ and derive some new ones; see Theorem~\ref{theo:properties_free_normal} and Proposition~\ref{prop:free_unitary_normal_conv_wigner}.
\end{itemize}

\subsection*{Notation}
Throughout the paper, $\bD = \{z\in \C:\, |z| <1\}$ denotes the open unit disk, $\bT = \{z\in \C:\, |z|=1\}$ the unit circle, and $\bH = \{\theta \in \C:\, \Im \theta >0\}$ the upper half-plane. The closures of $\bD$ and $\bH$ are denoted by $\overline\bD$ and $\overline \bH$, respectively. We write $a_n\sim b_n$ if $a_n/b_n\to 1$ as $n\to\infty$.  Weak and vague convergence of measures are denoted by $\overset{w}\longrightarrow$ and $\overset{v}\longrightarrow$, respectively.

\section{Main results}\label{sec:main_results}

\subsection{Empirical distribution of zeroes}
The \textit{empirical distribution of zeroes} of an algebraic polynomial $P_n(z)$ of degree $n$, i.e.\ the probability measure  assigning to each zero the same weight $1/n$, will be denoted by
\begin{equation}\label{eq:empirical_distr_zeroes_def}
\mu\lsem P_n \rsem := \frac 1n \sum_{\substack{z\in \C:\, P_n(z) = 0}} \delta_z.
\end{equation}
We agree that the roots are always counted with multiplicities.
It is well known, see, e.g, \cite{ullman0,gawronski,kornyik_michaletzky}, that the empirical distribution of zeroes of the classical Hermite polynomial $\He_n(z\sqrt {n})$ converges weakly to the Wigner distribution $\gamma_{0,2}$ with the density $x\mapsto \frac 1 {2\pi} \sqrt{4-x^2}$ on the interval $[-2,2]$, namely
$$
\mu\lsem \He_n(\,\cdot\, \sqrt {n}) \rsem = \frac 1n \sum_{\substack{z\in \C:\, \He_n(z) = 0}} \delta_{z/\sqrt n} \toweak \gamma_{0,2}.
$$
Given that the Wigner law is the analogue of the normal distribution w.r.t.\ the free additive convolution $\boxplus$, one may conjecture that the limiting empirical distribution of zeroes of the \textit{unitary} Hermite polynomials should be related to the analogue of the normal distribution w.r.t.\ the free \textit{multiplicative} convolution $\boxtimes$.   We shall confirm this intuition. We begin by recording the following important property.

\begin{lemma}\label{lem:no_zeroes_in_D}
All zeroes of the polynomial $H_n(z;\sigma^2)$ are located on the unit circle $\bT = \{|z|=1\}$.  %for all $n\in \N$ and $\sigma^2>0$.
\end{lemma}
\begin{proof}
The claim is a special case of the Lee--Yang theorem; see~\cite[Appendix~II]{lee_yangII} (where one takes $x_{\alpha \beta} := \eee^{-\sigma^2/2}$ for all $\alpha, \beta =1,\ldots,n$) or~\cite[Section~5.1]{ruelle_book}.
Alternatively, the claim can be deduced from the P\'olya--Benz theorem~\cite[Theorem~1.2]{aleman_beliaev_hedenmalm} applied to the periodic function $f(\theta) =  (\sin \frac \theta 2)^{n}$ and the differential operator $\exp\{-\frac{1}{2}\sigma^2 \partial_\theta^2\}$ (see the remarks preceding Corollary~1.3 in~\cite{aleman_beliaev_hedenmalm} regarding applicability to non-polynomials). The P\'olya--Benz theorem implies that all zeroes of $\exp\{-\frac{1}{2}\sigma^2 \partial_\theta^2\}(\sin \frac \theta 2)^{n}$ are real.  Recalling~\eqref{eq:exp_diff_sin_theta_power} completes the proof.
\end{proof}

\begin{figure}[t]
	\centering
	\includegraphics[width=0.22\columnwidth]{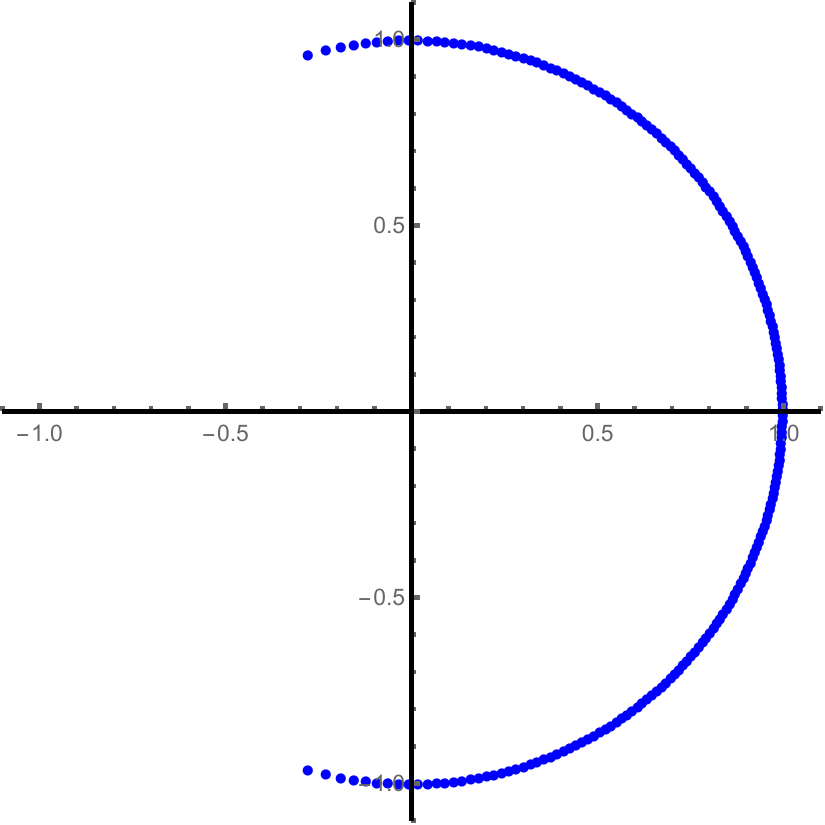}
	\includegraphics[width=0.22\columnwidth]{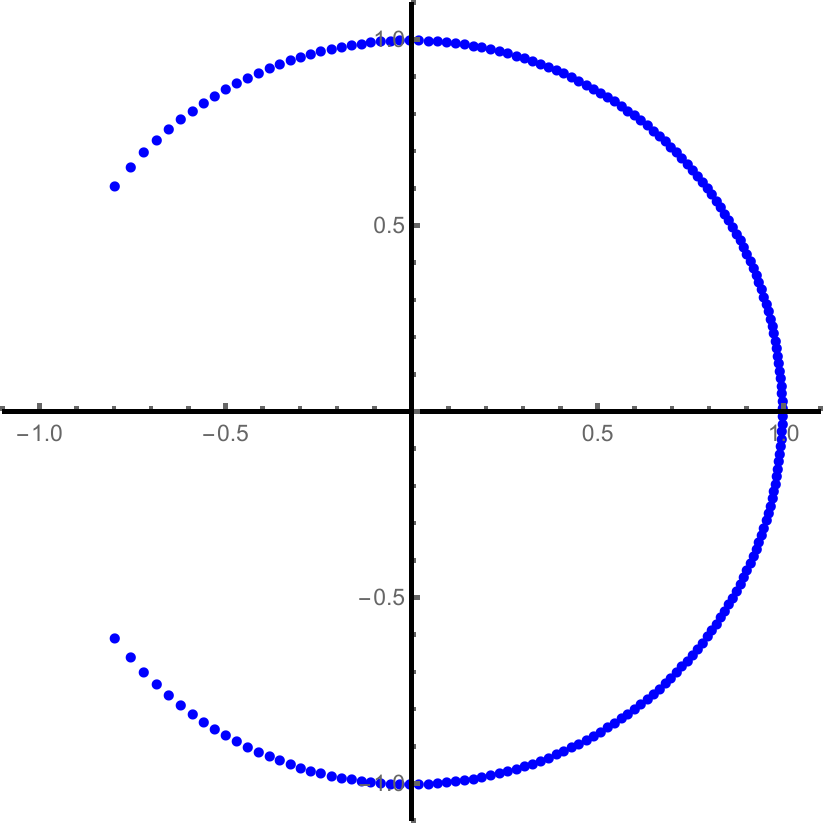}
	\includegraphics[width=0.22\columnwidth]{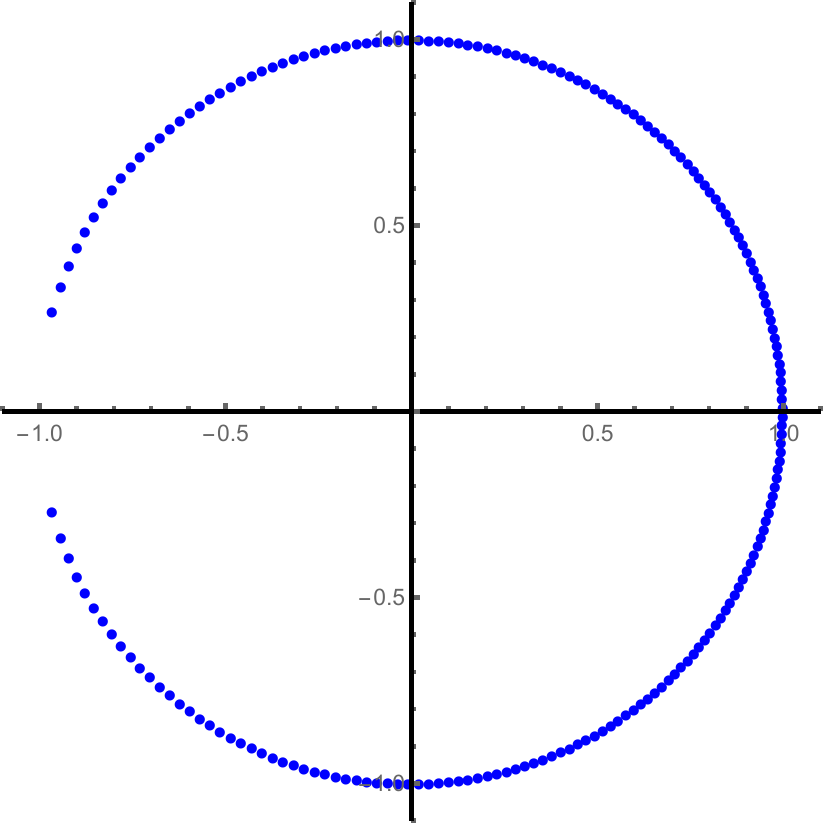}
	\includegraphics[width=0.22\columnwidth]{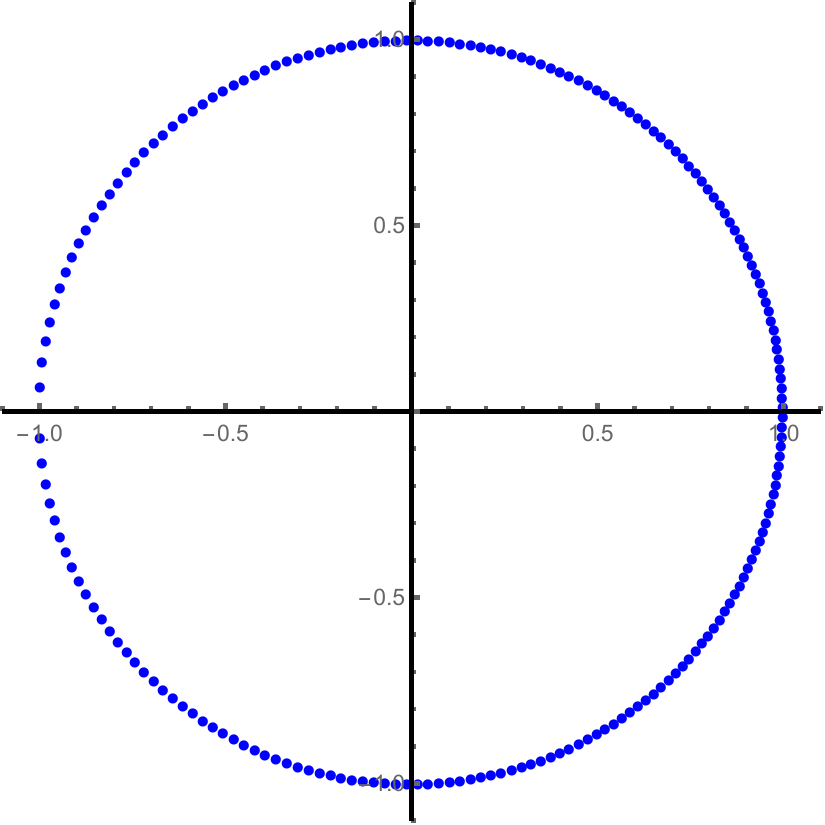}
	\caption{Zeroes of the unitary Hermite polynomials $H_n(z; \sigma^2/n)$ with $\sigma^2= 1,2,3,4$. The degree is $n=200$.}
\label{fig:zeroes_hermite}
\end{figure}

The empirical distribution of zeroes of $H_n(z; \sigma^2/n)$ will be denoted by
\begin{equation}\label{eq:Psi_mu_n_sigma^2}
\mu_{n;\sigma^2} := \mu\lsem H_n(\,\cdot\,; \sigma^2/n) \rsem =  \frac 1n \sum_{z\in \bT:\,  H_n(z; \sigma^2/n) = 0} \delta_z.
\end{equation}
\begin{theorem}\label{theo:weak_conv_zeroes_to_free_normal}
Fix some $\sigma^2>0$. Then, as $n\to\infty$, the probability measures $\mu_{n; \sigma^2}$ converge weakly on $\bT$ to the free unitary normal distribution $\fnorm_{\sigma^2}$ with parameter $\sigma^2$; see Section~\ref{sec:free_unitary_normal} for its definition and properties.
\end{theorem}

Attempting to prove Theorem~\ref{theo:weak_conv_zeroes_to_free_normal} by the method of moments leads to non-trivial combinatorics.  After the preprint version of this paper appeared, such proof has been given in~\cite{arizmendi_fujie_ueda_new_comb}.

\subsection{Asymptotics of unitary Hermite polynomials}
Our asymptotic results on the polynomials $H_n(z;\sigma^2/n)$ will be stated in terms of certain analytic function $\zeta_t(\theta)$ that satisfies
$$
\zeta_t(\theta) - t \tan \zeta_t(\theta) = \theta,
$$
where $t>0$ is a parameter and $\theta$ is a complex variable satisfying $\Im \theta >0$.  This function is related to the free unitary Poisson distribution, as has been shown in~\cite{kabluchko_rep_diff_free_poi}, and to the free unitary normal distribution, which will be demonstrated in Section~\ref{sec:free_unitary_normal} below.    The next theorem summarizes the main properties of this function; see~\cite[Section~4]{kabluchko_rep_diff_free_poi} for proofs and further properties.
\begin{theorem}\label{theo:riemann_surface_z_i}
Fix $t > 0$. Let $\bH := \{\theta\in \C:\, \Im \theta > 0\}$ be the upper half-plane. For every  $\theta\in \bH$,  the equation $\zeta - t \tan \zeta = \theta$ has a unique, simple solution $\zeta= \zeta_t(\theta)$ in $\bH$.  The function $\zeta_t:\bH\to \bH$ is analytic on $\bH$, admits a continuous extension to the closed upper half-plane $\overline \bH$, and satisfies
\begin{equation}\label{eq:zeta_properties}
\zeta_t(\theta + \pi ) = \zeta_t(\theta) + \pi,
\qquad
\zeta_t(-\bar \theta) = -\overline{\zeta_t(\theta)},
\qquad
\Im \zeta_t(\theta) > \Im \theta,
\end{equation}
for all $\theta\in \bH$. Locally uniformly in $x\in \bH$ we have
$$
\zeta_t(\theta) = \lim_{n\to\infty} \underbrace{(\theta + t \tan (\theta + t \tan (\ldots (\theta + t \tan x) \ldots )))}_{n \text{ iterations}},
\qquad
\theta \in \overline \bH.
$$
Finally, we have
\begin{equation}\label{eq:zeta_limit_behavior}
\zeta_{t} (\theta) -\theta \to  \ii t
\quad
\text{ as }
\quad
\Im \theta\to +\infty
\quad
\text{ uniformly in } \Re \theta\in \R.
\end{equation}
\end{theorem}

The next theorem is our second main result.
\begin{theorem}\label{theo:hermite_unitary_asympt}
Locally uniformly in $\theta \in \bH$ we have
\begin{align}
\lim_{n\to\infty} \frac 1n \log \frac{H_n(-\eee^{\ii \theta};\sigma^2/n)}{(-1)^n}
=
\log \left(1+\eee^{2\ii \zeta_{\sigma^2/4}(\theta/2)}\right) - \frac{\sigma^2}{2\left(1+ \eee^{-2 \ii \zeta_{\sigma^2/4}(\theta/2)}\right)^2},
\label{eq:log_asympt_H_n}\\
\lim_{n\to\infty} \frac 1n \, \frac{H_n'(-\eee^{\ii \theta}; \sigma^2/n)}{H_n(-\eee^{\ii \theta}; \sigma^2/n)}
=
\frac{- \eee^{-\ii \theta}}{1+ \eee^{-2 \ii \zeta_{\sigma^2/4}(\theta/2)}}
=
- \eee^{-\ii \theta} \left(\frac \ii 2 \tan \zeta_{\sigma^2/4} (\theta/2) + \frac 12\right).
\label{eq:asympt_H_n_log_der}
\end{align}
The logarithms in~\eqref{eq:log_asympt_H_n} are chosen such that $\log 1 = 0$ and all functions of the form $\log (\ldots)$ are continuous (and analytic) in $\theta \in \bH$.
\end{theorem}
\begin{remark}\label{rem:analytic_functions_well_defined}
We can consider all functions appearing in~\eqref{eq:log_asympt_H_n} and~\eqref{eq:asympt_H_n_log_der} as analytic functions of the variable $z:= -\eee^{\ii \theta}\in \bD$ including the value $z=0$. Firstly, $z\to 0$ is equivalent to $\Im \theta \to +\infty$. In this regime, \eqref{eq:zeta_limit_behavior} implies that $\zeta_{\sigma^2/4}(\theta/2) = (\theta/2)  + \ii (\sigma^2/4) + o(1)$ and, consequently,  $\Im \zeta_{\sigma^2/4}(\theta/2) \to +\infty$. It follows that the right-hand side of~\eqref{eq:log_asympt_H_n} converges to $0$, while the right-hand side of~\eqref{eq:asympt_H_n_log_der} converges to $-\eee^{-\sigma^2/2}$. Secondly, note that $\theta$ corresponding to a given $z\in \bD \backslash\{0\}$ is defined only up to a summand of the form $2\pi n$ with $n\in \Z$. Still, the right-hand sides of~\eqref{eq:log_asympt_H_n} and~\eqref{eq:asympt_H_n_log_der} stay invariant under the substitution $\theta \mapsto \theta + 2\pi n$ (since $\zeta_{\sigma^2/4}(\theta/2 + \pi n) = \zeta_{\sigma^2/4}(\theta/2) + \pi n$ by~\eqref{eq:zeta_properties}) and hence define analytic functions of $z\in \bD$.  Analogous observations apply to many similar functions below. Note that convergence in~\eqref{eq:log_asympt_H_n} and~\eqref{eq:asympt_H_n_log_der} stays  locally uniform in $z\in \bD$. For $z$ outside any small disk around $0$, this is stated in Theorem~\ref{theo:hermite_unitary_asympt}, while the rest follows  from Cauchy's integration formula.
\end{remark}

\begin{remark}
Theorem~\ref{theo:hermite_unitary_asympt} describes the asymptotics of $H_n(z;\sigma^2/n)$ for $|z|<1$. The asymptotics for $|z|>1$ can be derived from the identity $z^n H_n(1/z; \sigma^2/n) = (-1)^n H_n(z; \sigma^2/n)$ following from~\eqref{eq:hermite_poly_circ_def1}. On the circle $\{|z|=1\}$ one may expect an asymptotic result of Plancherel-Rotach type; see~\cite[Theorem~8.22.9]{szegoe_book} for the case of the classical Hermite polynomials.
\end{remark}

\subsection{Applications to the Curie--Weiss model}
We are now going to describe the global limiting distribution of zeroes of $Z_n(\beta, h)$, the partition function of the Curie--Weiss model defined in~\eqref{eq:curie_weiss_part_funct_def}. We consided the so-called Lee--Yang zeroes, that is we fix real $\beta>0$ and allow $h$ to be complex. By the Lee--Yang theorem, all zeroes are purely imaginary; see~\eqref{eq:curie_weiss_part_funct_reduction_hermite} and  Lemma~\ref{lem:no_zeroes_in_D}.  Observe also that $Z_n(\beta, h +\pi \ii) = \eee^{\pi \ii n} Z_n(\beta, h)$ by~\eqref{eq:curie_weiss_part_funct_def}  implying that the zeros are periodic with period $\pi \ii$.
\begin{corollary}\label{cor:curie_weiss_zeroes}
Fix $\beta>0$. For the partition function of the Curie--Weiss model, the following convergence holds vaguely on $\R$:
$$
\frac 1n \sum_{y\in \R:\,  Z_n(\beta, \ii y) = 0} \delta_{y} \tovague \nu_{\beta}.
$$
Here, $\nu_\beta$ is a measure on $\R$ which is invariant under the shifts $h\mapsto h + \pi \ell$, $\ell\in \Z$,  and is characterized by $\nu_\beta(A) = \fnorm_{4\beta} ( - \eee^{2 \ii A})$ for every Borel set $A\subset (-\frac{\pi}{2}, \frac{\pi}{2}]$, where  $\fnorm_{4\beta}$ is the free unitary normal distribution on the unit circle with parameter $\sigma^2 = 4\beta$; see Section~\ref{sec:free_unitary_normal}.
\end{corollary}
\begin{proof}
Let $f:\R \to \R$ be a continuous function with compact support.  Define $f^*(y):= \sum_{\ell\in \Z} f(y + \pi \ell)$ for $y\in (-\frac{\pi}{2}, \frac{\pi}{2}]$. Let also $\psi: \bT \to  (-\frac{\pi}{2}, \frac{\pi}{2}]$ be the inverse map of $y\mapsto -\eee^{2\ii y}$. Then, by~\eqref{eq:curie_weiss_part_funct_reduction_hermite},
$$
\frac 1n \sum_{y\in \R:\,  Z_n(\beta, \ii y) = 0} f(y)
=
\frac 1n \sum_{y\in (-\frac{\pi}{2}, \frac{\pi}{2}]:\,  H_n(-\eee^{2\ii y}, \frac{4\beta}{n}) = 0} f^*(y)
=
\frac 1n \sum_{z\in \bT:\,  H_n(z, \frac{4\beta}{n}) = 0} f^*(\psi(z)).
$$
By Theorem~\ref{theo:weak_conv_zeroes_to_free_normal}, the latter sum converges to
$\int_{\bT} f^*(\psi(z))\fnorm_{4\beta}(\dint z) = \int_{\R} f(y)\nu_{\beta}(\dint y)$, and the claim follows.
\end{proof}
\begin{remark}
The Lebesgue density of $\nu_\beta$ on $\R$ is given by $y\mapsto 2 f_{4\beta} (- \eee^{2\ii y}) = \frac{1}{\pi \beta} \Im \zeta_\beta(y)$,  where $f_{4\beta}$ is the function which will be discussed in Theorem~\ref{theo:properties_free_normal}. It follows  from this theorem that the support of $\nu_\beta$ is $\R$ for $\beta \geq 1$, while for $0<\beta<1$ the support is the union of the intervals
$$
\left[\frac{\pi}{2} - \arcsin  \sqrt \beta -  \sqrt{\beta - \beta^2}  + \pi \ell, \frac{\pi}{2} + \arcsin  \sqrt \beta +  \sqrt{\beta - \beta^2} + \pi \ell\right],
\qquad
\ell\in \Z.
$$
If $\beta$ increases from $0$ to $\infty$, then the support of $\nu_\beta$ hits the real axis  at $\beta=1$, which is well known to be the point of phase transition for the Curie--Weiss model.
\end{remark}

In the next result we compute the free energy of the Curie--Weiss model in the complex $h$-plane excluding the imaginary axis.
\begin{corollary}\label{cor:curie_weiss_free_energy}
Let $\beta >0$ and $h\in \C$ with $\Re h >0$. For the partition function of the Curie--Weiss model defined in~\eqref{eq:curie_weiss_part_funct_def} we have
$$
\lim_{n\to\infty} \frac 1n \log Z_n(\beta, h)
=
\lim_{n\to\infty} \frac 1n \log Z_n(\beta, -h)
=
\frac \beta 2 + h + \log \left(1+\eee^{2\ii \zeta_{\beta}(\ii h)}\right) - \frac{2\beta}{\left(1+ \eee^{-2 \ii \zeta_{\beta}(\ii h)}\right)^2}.
$$
\end{corollary}
\begin{proof}
Recall that $Z_n(\beta, h) = Z_n(\beta, -h)$ is given by~\eqref{eq:curie_weiss_part_funct_def} and~\eqref{eq:curie_weiss_part_funct_reduction_hermite} and apply Theorem~\ref{theo:hermite_unitary_asympt} with $\sigma^2= 4\beta$ and $\theta = 2\ii h$.
\end{proof}

\subsection{Expected characteristic polynomial of the Brownian motion on unitary matrices}
It is well known~\cite[Section~6.1.2]{potters_bouchaud_book} that the expected characteristic polynomial of an $n\times n$ Wigner random matrix coincides with the $n$-th Hermite polynomial. Formulas of this type  go back  to Heine~\cite[Eqn.~(2.2.11) on p.~27]{szegoe_book}. For this and analogous results on several other types of random matrices including the Wishart matrices whose expected characteristic polynomials are the Laguerre polynomials we refer to~\cite[Sections~6.2, 6.3]{potters_bouchaud_book}, \cite{aomoto}, \cite[Chapter~9]{edelman_phd}, \cite[Theorem~4.1]{dumitriu_edelman},  \cite[Eqn.~(15)]{brezin_hikami}, \cite[Proposition~12]{diaconis_gamburd}, \cite[Proposition~11]{forrester_gamburd}, \cite[Theorem~1.1]{akemann_goetze_neuschel}.

In this section we prove a similar result on unitary Hermite polynomials by relating them to the expected characteristic polynomials of the random matrices obtained by running a Brownian motion on the unitary group $U(n)$. More precisely, we consider the unitary group $U(n)$ as a compact Riemannian manifold endowed with the Riemannian metric induced by its natural embedding into $\C^{n\times n}\equiv \R^{2n^2}$.
On the Lie algebra $\mathfrak{u}(n)= \{A\in \text{Mat}_{n\times n}(\C): A^* = -A\}$ (which can be identified with the tangent space of $U(n)$ at  the identity matrix $I_n$) the scalar product takes the form $\langle A,  B\rangle = \Tr (A B^*) = - \Tr (AB)$.

Let now $(U_n(t))_{t\geq 0}$ be the Brownian motion on the unitary group $U(n)$ starting at the identity matrix $I_n$ at time $t=0$. The eigenvalues $\lambda_1(t), \ldots, \lambda_n(t)$ of the unitary random matrix $U_n(t)$ represent a special case of Dyson's Brownian motions~\cite[Section~III]{dyson} on the circle; see also~\cite{hobson_werner,cepa_lepingle_circle,biane_matrixvalued,hall_notes,buijsman,forrester2024dip} for further information on this process.

To define Dyson's Brownian motions on the circle, fix parameters $n\in \N$ and $\lambda>0$. Let $(B_1(t))_{t\geq 0},\ldots,(B_n(t))_{t\geq 0}$ be $n$ independent standard Brownian motions on $\R$. We are interested in real-valued stochastic processes  $X_1(t) \leq \ldots \leq  X_n(t)$, defined for $t\geq 0$ and solving stochastic differential equations
\begin{equation}\label{eq:diffusions_singular_drifts}
\dint X_j = \dint B_j + \lambda \cdot \bigg(\sum_{\substack{k\in \{1,\ldots,n\}\\ k\neq j}} \cot \frac{X_j -X_k}{2}\Bigg) \dint t,
\qquad j=1,\ldots, n,
\end{equation}
with the initial condition $X_1(0)= \ldots = X_n(0) = 0$. If $\lambda=1/2$, then we can identify $\eee^{\ii X_1(t)},\ldots, \eee^{\ii X_n(t)}$ with the eigenvalues $\lambda_1(t), \ldots, \lambda_n(t)$ of the unitary random matrix $U_n(t)$.
More precisely, it is known that the measure-valued process $(\sum_{\ell=1}^n \delta_{\eee^{\ii X_\ell(t)}})_{t\geq 0}$ has the same distribution as the process $(\sum_{\ell=1}^n \delta_{\lambda_\ell(t)})_{t\geq 0}$.

We are interested in the following polynomial in $x$ which, for $\lambda=1/2$, reduces to  the characteristic polynomial of $U_n(t)$:
$$
P_{n,\lambda}(x; t) := \prod_{j=1}^n \left(x- \eee^{\ii X_j(t)}\right).
$$
\begin{theorem}\label{theo:expected_char_poly_unitary_BM}
For every $\lambda>0$, $n\in \N$,  $t\geq 0$ and $x\in \C$ we have
$$
\E P_{n,\lambda}(x;t)  = \sum_{j = 0}^n(-1)^{n-j} \binom n j \eee^{ - \frac 12 (n-j) t - \lambda j (n-j)  t}  x^{j}
=
\eee^{-\frac 12 nt} H_n(\eee^{t/2} x; 2\lambda t).
$$
\end{theorem}
\begin{proof}
To simplify the notation, we shall usually suppress the dependence of quantities under consideration on $n$ and $\lambda$. Let $e_\ell(t)$ be the $\ell$-th elementary symmetric polynomial of $\eee^{\ii X_1(t)},\ldots, \eee^{\ii X_n(t)}$, that is
$$
e_\ell(t) = \sum_{1\leq j_1 < \ldots < j_\ell \leq n} \eee^{\ii X_{j_1}(t) +\ldots + \ii X_{j_\ell}(t)},
\qquad
\ell = 1,\ldots, n,
$$
Put also $e_0(t) = 1$.
Since $P_{n,\lambda}(x;t) = \sum_{\ell = 0}^n (-1)^\ell e_\ell(t) x^{n-\ell}$ by Vieta's formula, it suffices to show that for all $\ell \in \{1,\ldots, n\}$ we have
$$
\E e_\ell(t) = \binom n\ell \eee^{-\frac 1 2 \ell t -  \lambda \ell (n-\ell) t}.
$$

To this end, we shall derive stochastic differential equations satisfied by $e_\ell(t)$. Using the It\^{o} formula, see, e.g.,~\cite[Chapter~IV, Theorem~(3.3)]{revuz_yor_book}, we have
$$
\dint e_\ell
=
\sum_{1\leq j_1 < \ldots < j_\ell \leq n} \dint \left(\eee^{\ii X_{j_1} +\ldots + \ii X_{j_\ell}}\right)
=
\sum_{1\leq j_1 < \ldots < j_\ell \leq n}
\left(\sum_{s=1}^\ell \ii  \eee^{\ii X_{j_1} +\ldots + \ii X_{j_\ell}}\dint X_{j_s}
-\frac {\ell}2  \eee^{\ii X_{j_1} +\ldots + \ii X_{j_\ell}} \dint t
\right).
$$
Write $V(x) := \lambda \cot  \frac x2$. Recalling~\eqref{eq:diffusions_singular_drifts}, we obtain
\begin{equation}\label{eq:elem_symm_SDE1}
\dint e_\ell = \sum_{1\leq j_1 < \ldots < j_\ell \leq n} \ii \eee^{\ii X_{j_1} +\ldots + \ii X_{j_\ell}} \left(\dint B_{j_1} +\ldots + \dint B_{j_\ell} \right) - \frac \ell 2 e_\ell \dint t
+
R \dint t
\end{equation}
with
\begin{align*}
R
&:=
\sum_{1\leq j_1 < \ldots < j_\ell \leq n} \ii \eee^{\ii X_{j_1} +\ldots + \ii X_{j_\ell}}  \Bigg( \sum_{\substack{m\in \{1,\ldots, n\}\\ m\neq j_1}} V(X_{j_1} - X_m) + \ldots + \sum_{\substack{m\in \{1,\ldots, n\}\\ m\neq j_\ell}} V(X_{j_\ell} - X_m)\Bigg)\\
&=
\sum_{1\leq j_1 < \ldots < j_\ell \leq n} \ii \eee^{\ii X_{j_1} +\ldots + \ii X_{j_\ell}}  \Bigg( \sum_{m\in \{1,\ldots, n\}\backslash \{j_1,\ldots, j_\ell\}} \Big(V(X_{j_1} - X_m) + \ldots +  V(X_{j_\ell} - X_m)\Big) \Bigg),
\end{align*}
where in the second line we used that $V(-x) = -V(x)$. After some re-indexing, we can write
\begin{align*}
R
&=
\frac{\ii}{(\ell+1)!}\sum_{\substack{k_0,\ldots, k_\ell\in \{1,\ldots, n\}\\\text{pairwise distinct}}}\eee^{\ii X_{k_0} + \ii X_{k_1} + \ldots + \ii X_{k_\ell}} \sum_{\substack{s,p\in \{0,\ldots,\ell\}\\s\neq p}} \frac{V(X_{k_p} - X_{k_s})}{\eee^{\ii X_{k_s}}}\\
&=
\frac{\ii}{2 (\ell+1)!}\sum_{\substack{k_0,\ldots, k_\ell\in \{1,\ldots, n\}\\\text{pairwise distinct}}}\eee^{\ii X_{k_0} + \ii X_{k_0} + \ldots + \ii X_{k_\ell}} \sum_{\substack{s,p\in \{0,\ldots,\ell\}\\s\neq p}} \left(\frac{V(X_{k_p} - X_{k_s})}{\eee^{\ii X_{k_s}}} + \frac{V(X_{k_s} - X_{k_p})}{\eee^{\ii X_{k_p}}}\right).
\end{align*}
To simplify the expression in the brackets, note that
$$
\frac{V(x - y)}{\eee^{\ii y}} + \frac{V(y - x)}{\eee^{\ii x}}
=
\lambda  \cot \left(\frac {x-y}{2}\right) \left(\eee^{\ii (x-y)} - 1\right) \eee^{-\ii x}
=
\ii \lambda \left(\eee^{- \ii x} + \eee^{- \ii y}\right).
$$
It follows that
\begin{align*}
R
&=
\frac{-\lambda}{2 (\ell+1)!}\sum_{\substack{k_0,\ldots, k_\ell\in \{1,\ldots, n\}\\\text{pairwise distinct}}}\eee^{\ii X_{k_0} + \ii X_{k_1} + \ldots + \ii X_{k_\ell}} \sum_{\substack{s,p\in \{0,\ldots,\ell\}\\s\neq p}} \left(\eee^{ - \ii X_{k_s}} + \eee^{-\ii X_{k_p}}\right)\\
&= - \lambda \ell (n-\ell) e_\ell(t).
\end{align*}
To justify the last identity, observe that the double sum in the first line must be a multiple of $e_\ell$ for symmetry reasons and that it contains $(n)_{\ell+1} \cdot (\ell+1) \ell \cdot 2$ summands, while $e_\ell$ contains $\binom n \ell$ summands. Taking the quotient of these two numbers, it follows that the double sum in the first line equals $2(\ell+1)! \ell(n-\ell) e_{\ell}$.

Finally, recalling~\eqref{eq:elem_symm_SDE1}, we arrive at the stochastic differential equation
\begin{equation}\label{eq:elem_symm_SDE2}
\dint e_\ell = - \left(\frac \ell 2 + \lambda \ell (n-\ell)\right) e_\ell \dint t + \sum_{j=1}^n  \ii \eee^{\ii X_j} e_{\ell-1}^{(j)} \dint B_j,
\end{equation}
where $e_{\ell-1}^{(j)}$ is the $(\ell-1)$-st elementary symmetric polynomial of $\eee^{\ii X_1},\ldots,\eee^{\ii X_{j-1}}$, $\eee^{\ii X_{j+1}}, \ldots, \eee^{\ii X_{n}}$ (excluding $\eee^{\ii X_{j}}$). From the It\^{o} formula it follows that $\eee^{\frac 1 2 \ell t + \lambda \ell (n-\ell) t} e_\ell(t)$ is a martingale.
Recalling that $e_\ell(0) = \binom n \ell$ we conclude that
$$
\E e_\ell (t) = \binom n\ell \eee^{-\frac 1 2 \ell t - \lambda \ell (n-\ell) t},
$$
and the proof is complete.
\end{proof}

Following a suggestion of an anonymous referee let us extend the above result to the Brownian motion $(V U_n(t))_{t\geq 0}$ on the unitary group  $U(n)$ starting from an arbitrary unitary matrix $V\in U(n)$.
\begin{corollary}\label{cor:unitary_BM}
For every unitary matrix $V\in U(n)$ and all $t\geq 0$ and $z\in \C$ we have
\begin{equation}\label{eq:Q_n_def_char_poly_BM_unitary_group}
Q_n(z; t ) = \E \det(z I_n - V U_n(t))
=
\det (z I_n-V) \boxtimes_n \eee^{-\frac 12 n t} H_n(\eee^{t/2}z; t).
\end{equation}
Also, the following PDE holds:
\begin{equation}\label{eq:PDE_unitary_BM}
\partial_t Q_n(z;t) = - \frac 12 (z \partial_z + 1) (n-z\partial_z) Q_n(z;t),\qquad  Q_n(z;0) = \det (z I_n-V).
\end{equation}
\end{corollary}
\begin{proof}
It is possible to prove~\eqref{eq:Q_n_def_char_poly_BM_unitary_group} by changing the initial condition for $e_\ell(t)$ in the proof of  Theorem~\ref{theo:expected_char_poly_unitary_BM}, however, we find it more instructive to give a proof that uses finite free probability.
Let $W$ be a Haar-distributed random matrix in $U(n)$ which is independent of everything else. Then, $V U_n(t)$ has the same law as $V W^{-1} U_n(t) W$. Hence,
\begin{align*}
\E \det(z I_n - V U_n(t))
&=
\E [\det(z I_n - V W^{-1} U_n(t) W)]\\
&=
\E [\E [\det(z I_n - V W^{-1} U_n(t) W) \,|\, U_n(t)]].
\end{align*}
In the conditional expectation, $U_n(t)$ is fixed and the integration is over the distribution of $W$, i.e.\ the Haar distribution on $U(n)$. By Theorem~1.5 from~\cite{marcus_spielman_srivastava}, the conditional expectation can be expressed using $\boxtimes_n$, which leads to
\begin{align*}
\E \det(z I_n - V U_n(t))
&=
\E [\det (z I_n-V) \boxtimes_n \det (z I_n-U_n(t))]\\
&=
\det (zI_n-V) \boxtimes_n \E \det (z I_n-U_n(t))\\
&=
\det (z I_n-V) \boxtimes_n \eee^{-\frac 12 n t} H_n(\eee^{t/2}z; t).
\end{align*}
We used the bilinearity of $\boxtimes_n$ and Theorem~\ref{theo:expected_char_poly_unitary_BM} with $\lambda = 1/2$. The proof of~\eqref{eq:Q_n_def_char_poly_BM_unitary_group} is complete.

Since $\boxtimes_n$ commutes with $z\partial_z$, see~\cite[p.~810]{marcus_spielman_srivastava}, it suffices to prove the PDE~\eqref{eq:PDE_unitary_BM} for $Q_n(z; t ) = \eee^{-\frac 12 n t} H_n(\eee^{t/2}z; t)$ with the initial condition $Q_n(z;0) = (z-1)^n$, which corresponds to the case when $V= I_n$.
Using that $(z \partial_z +1) (n- z\partial_z)  z^j = (j+1)(n-j) \cdot  z^j$, we deduce
\begin{multline*}
\exp\left\{-\frac t2 (z \partial_z+1) (n- z\partial_z)\right\} (z-1)^n
=
\sum_{j=0}^n (-1)^{n-j} \binom nj \exp\left\{-\frac t2  (z \partial_z+1) (n- z\partial_z)\right\} z^j
\\
=
\sum_{j=0}^n (-1)^{n-j} \binom nj  \exp\left\{ - \frac {t}{2} (j+1) (n-j)\right\} z^j
= \eee^{-\frac 12 n t} H_n(\eee^{t/2}z; t),
\end{multline*}
where we used ~\eqref{eq:hermite_poly_circ_def1} in the last step.
\end{proof}

\begin{remark}
The operator $(z \partial_z + 1) (n-z\partial_z) = -z^2 \partial_z^2 + (n-2) z \partial_z + n$ is (up to sign) the same as the one appearing in Conjecture 2.14 and Proposition 2.15 of Hall and Ho~\cite{hall_ho}.
\end{remark}

Biane~\cite{biane,biane_segal_bargmann} proved that, as $n\to\infty$, the process $(U_n(t/n))_{t\geq 0}$ converges (in a suitable sense) to the free unitary Brownian motion. In particular, by~\cite[Theorem~1]{biane}, the spectral distribution of $U_n(t/n)$ converges weakly to the free unitary normal distribution $\fnorm_{t}$ (making the appearance of this distribution in Theorem~\ref{theo:weak_conv_zeroes_to_free_normal} quite natural); see also~\cite[Section~3.3]{cabanal_duvillard_guionnet} for related large deviation results. Exact combinatorial formulas for moments of the form $\E [\Tr (U_n^{m_1}(t))\ldots \Tr (U_n^{m_r}(t))] $ have been derived in~\cite{levy_schur_weyl}.

\subsection{The action of the backward heat flow on the roots}
Consider a sequence of polynomials (or trigonometric polynomials) of increasing degrees whose empirical distributions of roots approach some  probability measure. One may ask what happens to the asymptotic distribution of roots if we apply to these polynomials certain operator. One special case, in which the operator is the repeated differentiation, has been studied in~\cite{steinerberger_real,steinerberger_conservation,steinerberger_free,orourke_steinerberger_nonlocal,hoskins_kabluchko,bogvad_etal,kiselev_tan,arizmendi_garza_vargas_perales,kabluchko_rep_diff_free_poi,galligo}.
For trigonometric polynomials, it has been shown in~\cite{kabluchko_rep_diff_free_poi} that, on the level of roots, the repeated differentiation induces the free unitary Poisson process.
In his blog, Tao~\cite{tao_blog1,tao_blog2} discusses the evolution of zeroes of a polynomial which undergoes a (backward) heat flow. As this paper was almost complete, Jonas Jalowy brought to our attention the recent preprint by Hall and Ho~\cite{hall_ho} who studied the action of the backward heat flow on the characteristic polynomials of the Ginibre matrices (whose eigenvalues obey the circular law).
We shall consider two settings: algebraic polynomials and trigonometric polynomials, both with real roots, and show that the backward heat flow induces free (additive or unitary) Brownian motion on the level of roots.

\subsubsection{Heat flow acting on algebraic polynomials}
Let $(P_n(z))_{n\in \N}$ be a sequence of algebraic polynomials from $\R[z]$. We suppose that $P_n(z) = \sum_{j=0}^n a_{j:n}z^j$ is real-rooted (that is, it has only real roots) and that all roots are contained in some fixed interval $[-C,C]$ with $C$ not depending on $n$. Moreover, we suppose that the empirical distribution of roots of $P_n$ converges weakly to some probability measure $\mu$ on $[-C,C]$, that is
\begin{equation}\label{eq:empirical_distr_zeroes_conv_algebraic}
\mu\lsem P_n \rsem = \frac 1n \sum_{\substack{z\in \R:\, P_n(z) = 0}} \delta_z \toweak \mu.
\end{equation}
The roots are counted with multiplicities, as always. We are interested in the action which the backward heat flow induces on the roots of $P_n$, in the large $n$ limit.  More precisely, we consider the heat equation on the real line with initial condition given by $P_n(z)$:
\begin{equation}\label{eq:heat_equation_algebraic}
\partial_t g_{n}(z; t) = \frac 12 \partial^2_z g_{n}(z;t),
\qquad
g_n(z; 0) = P_{n}(z),
\qquad
z\in\R,\; t\in \R.
\end{equation}
The solution is explicit and can be written as
\begin{equation}\label{eq:heat_eq_solution_alg}
g_{n} (z; -s)
=
\eee^{-\frac s2  \partial_z^2} P_{n}(z)
=
\sum_{j=0}^n a_{j:n} \eee^{-\frac s2  \partial_z^2} z^j
=
\sum_{j=0}^n  a_{j:n} \He_j\left(\frac{z}{\sqrt s}\right) s^{j/2},
\qquad
z\in \R, \; s\in \R,
\end{equation}
where $\He_j(z)$ is the $j$-th probabilist Hermite polynomial defined by~\eqref{eq:hermite_poly_def} or~\eqref{eq:eq:hermite_exp_on_z^n_with_sigma}.
Note that the solution exists both for positive and negative times since the term $\He_j(z/\sqrt s) s^{j/2}$ does not contain fractional powers of $s$, see~\eqref{eq:hermite_poly_def}, and makes sense irrespective of the sign of $s$. Moreover, for every $s\in \R$, the function  $z\mapsto g_n(z; -s)$ is a polynomial. In the sequel, we shall focus on the case $s>0$, which corresponds to the backward heat equation.  It is known to be ill-posed for initial conditions more general than polynomials. Since we assume that the initial condition $P_n(z)$ is real-rooted, the polynomials $g_n(z; -s)$ remain real-rooted for all $s\geq 0$ by the P\'olya--Benz theorem; see~\cite{benz} or~\cite[Theorem~1.2]{aleman_beliaev_hedenmalm}. Another proof of this fact can be found~\cite{tao_blog1}.
\begin{theorem}\label{theo:heat_flow_zeroes_algebraic}
Fix $r>0$. %Assume  that~\eqref{eq:empirical_distr_zeroes_conv_algebraic} holds and that $\mu$ is compactly supported.
Under the above assumptions, the empirical distribution of zeroes of the polynomials $g_{n}(z; -r^2/n)$ converges weakly (as $n\to\infty$) to the free additive convolution $\mu \boxplus \gamma_{0, 2r}$ of $\mu$ and the Wigner semicircle distribution $\gamma_{0,2r}$ with density $x\mapsto \frac 1 {2\pi r^2} \sqrt{4r^2-x^2}$ on the interval $[-2 r,2r]$.
\end{theorem}
To prove this theorem, we shall use a result from finite free probability~\cite{marcus,marcus_spielman_srivastava}.
Recall from~\eqref{eq:finite_free_add_conv_def} the definition of the finite free additive convolution  $\boxplus_{n}$. It is known from~\cite[Corollary~5.5 and Theorem~5.4]{arizmendi_perales}, see also~\cite[Theorem~4.3]{marcus}, that, as $n\to\infty$,  the finite free additive convolution $\boxplus_n$ approaches the free additive convolution $\boxplus$ in the following sense.
\begin{proposition}\label{prop:finite_free_add_conv_to_free_add}
Let $(p_n)_{n\in \N}$ and $(q_n)_{n\in \N}$ be sequences of polynomials  in $\mathbb R[z]$ whose roots are contained in some fixed interval $[-C,C]$. Suppose that $\deg p_n = \deg q_n = n$ and the empirical distributions of zeroes of $p_n$ and $q_n$ converge weakly to certain probability measures $\nu$ and $\rho$ on $[-C,C]$.   Then, the empirical distribution of zeroes of $p_n\boxplus_n q_n$ (which is also real-rooted by~\cite[Theorem~1.3]{marcus_spielman_srivastava}) converges weakly to $\nu\boxplus \rho$.
\end{proposition}
Indeed,  $\mu\lsem p_n \rsem\to \nu$ and $\mu\lsem q_n \rsem\to \rho$ (weakly) implies that the moments of $\mu\lsem p_n \rsem$ and $\mu\lsem q_n \rsem$ converge to the corresponding moments of $\nu$ and $\rho$ (since all measures are concentrated on a fixed interval).  By the results cited above, this implies that the moments of the empirical distribution of zeroes of $p_n\boxplus_n q_n$ converge to those of $\nu\boxplus \rho$. This implies that the probability measures $\mu\lsem p_n \boxplus_n q_n \rsem$ (which are concentrated on $[-2C,2C]$ by~\cite[Theorem 1.3]{marcus_spielman_srivastava}) converge weakly to $\nu \boxplus \rho$.
%Let $p_n$ and $q_n$ be polynomials in $\mathbb C[z]$ with $\deg p_n = \deg q_n \to\infty$ and suppose that the moments of the empirical distributions of zeroes converge to the moments of certain probability measures $\mu$ and $\nu$ on $\C$ (having all moments finite); see~\eqref{eq:moments_empirical_root_distr_converge}. Then, the moments of the empirical distribution of zeroes of $p_n\boxtimes_n q_n$ converge to those of $\mu\boxtimes \nu$.
\begin{proof}[Proof of Theorem~\ref{theo:heat_flow_zeroes_algebraic}.]
%We can now prove Theorem~\ref{theo:heat_flow_zeroes_algebraic}.
It is known~\cite{marcus_spielman_srivastava} that $\boxplus_n$ commutes with differentiation in the sense that $\partial_z (p(z) \boxplus_n q(z))= (\partial_z p(z)) \boxplus_n q(z)$ for arbitrary polynomials $p(z)$ and $q(z)$ of degree at most $n$. By bilinearity of $\boxplus_n$ it follows that it also commutes with the operator $\exp\{-\frac s2  \partial_z^2\}$, for all $s\in \R$.  Hence,
\begin{align*}
g_{n} (z; -s)
&
=
\eee^{-\frac s2  \partial_z^2} P_{n}(z)
=
\eee^{-\frac s2  \partial_z^2} \left( P_{n} (z) \boxplus_n z^n \right)\\
&=
P_{n}(z)\boxplus_n \left(\eee^{-\frac s2  \partial_z^2}  z^n\right)
=
P_n(z) \boxplus_n \He_n\left(\frac{z}{\sqrt s}\right) s^{n/2},
\end{align*}
where we used~\eqref{eq:eq:hermite_exp_on_z^n_with_sigma} in the last equality.
%After some elementary  transformations using~\eqref{eq:heat_eq_solution_alg} and~\eqref{eq:hermite_poly_def}, one finds that
%$$
%g_{n} (z; -s) = \sum_{j=0}^n  a_{j:n} \He_j\left(\frac{z}{\sqrt s}\right) s^{j/2} = P_n(z) \boxplus_n %\He_n\left(\frac{z}{\sqrt s}\right) s^{n/2}, \qquad s\in \R.
%$$
We take $s = r^2/ n$ and let $n\to\infty$.  By a classical result, the empirical distribution of zeroes of the Hermite polynomial $\He_n(z\sqrt {n}/r)$ converges weakly to the Wigner distribution with the density $x\mapsto \frac 1 {2\pi r^2} \sqrt{4r^2-x^2}$ on the interval $[-2r,2r]$; see, e.g., \cite{ullman0,gawronski,kornyik_michaletzky}. Applying Proposition~\ref{prop:finite_free_add_conv_to_free_add} completes the proof.
\end{proof}
\begin{remark}
Tao~\cite{tao_blog1} derived the following system of differential equations satisfied by the roots $z_{1:n}(s),\ldots, z_{n:n}(s)$ of the polynomial $g_n(z; -s)$:
\begin{equation}\label{eq:dyson_without_noise}
\partial_s z_{i:n} (s) = \sum_{\substack{k\in \{1,\ldots, n\}\\ k\neq i}} \frac{1}{z_{i:n}(s) - z_{k:n}(s)},
\qquad
i=1,\ldots, n,
\end{equation}
and proved that the solutions are well-defined for $s\geq 0$ but may run into a singularity for $s\leq 0$.
These equations are Dyson's Brownian motions with vanishing variance (or with $\beta= \infty$); see~\cite[Theorem~4.3.2]{anderson_etal_book}. Therefore, one can view Theorem~\ref{theo:heat_flow_zeroes_algebraic} as the $\beta= \infty$ case of~\cite[Proposition~4.3.10]{anderson_etal_book}. In this form, it has been established in~\cite[Theorem~1.1]{voit_werner2}.
For other results relating Dyson's Brownian motions to free convolutions we refer to~\cite{andraus2012,dumitriu_edelman_large_beta,gorin_kleptsyn,voit_woerner1,voit_werner2}.
\end{remark}

\subsubsection{Heat flow acting on trigonometric polynomials}
Let us now state analogous results for trigonometric polynomials.
For every even number $n=2d$, $d\in \N$, let $T_{n} (\theta) := \sum_{\ell= -d}^d c_{\ell:n} \eee^{\ii \ell \theta}$ be a trigonometric polynomial with complex coefficients $c_{\ell:n}$ satisfying  $c_{-\ell:n} = \overline{c_{\ell:n}}$ for all $\ell\in \{-d,\ldots, d\}$ (meaning that $T_n(\theta)$ takes real values for real $\theta$). Moreover, assume that $T_{n}(\theta)$ is real-rooted meaning that it has $n$ real roots $\theta_{1; n},\ldots, \theta_{n; n}$ (counting multiplicities) and that its empirical distribution of zeroes converges weakly as $n\to\infty$. The latter assumption will be written in the following form:
\begin{equation}\label{eq:empirical_distr_zeroes_conv}
\nu\lsem T_{n}\rsem := \frac {1}{n} \sum_{j=1}^{n} \delta_{\eee^{\ii \theta_{j;n}}} \toweak \nu
\end{equation}
weakly on the unit circle $\bT$, for some probability measure $\nu$ on $\bT$. Consider the heat equation on the real line with periodic initial condition given by $T_n(\theta)$:
\begin{equation}\label{eq:heat_equation}
\partial_t f_{n}(\theta; t) = \frac 12 \partial^2_\theta f_{n}(\theta;t),
\qquad
f_n(\theta; 0) = T_{n}(\theta),
\qquad
\theta\in \R, \; t\in \R.
\end{equation}
Its solution is explicit and can be written as
\begin{equation}\label{eq:heat_eq_solution}
f_{n} (\theta; t)
=
\eee^{\frac t2  \partial_\theta^2} T_{n}(\theta)
=
\sum_{\ell= -d}^d c_{\ell:n} \eee^{\frac t2  \partial_\theta^2} \eee^{\ii \ell \theta}
=
\sum_{\ell = -d}^d  c_{\ell:n} \eee^{-\frac t2 \ell^2} \eee^{\ii \ell \theta}.
\end{equation}
For every $t\in \R$, the function  $\theta \mapsto f_{n} (\theta; t)$ is a trigonometric polynomial. In particular, the solution makes sense both for $t\geq 0$ and for $t<0$. In the sequel we shall focus on the latter case, which corresponds to the backward heat equation.
For $s>0$, the trigonometric polynomial $f_n(\theta; -s)$ remains real-rooted.  This claim can be deduced from the P\'olya--Benz theorem~\cite[Corollary~1.3]{aleman_beliaev_hedenmalm}. Another proof can be found in~\cite{tao_blog2} (where differential equations analogous to~\eqref{eq:dyson_without_noise} in the circular setting are derived; see also~\cite{cepa_lepingle_circle,hobson_werner} for the corresponding stochastic differential equations).

\begin{theorem}\label{theo:heat_flow_zeroes}
Take some $\sigma^2>0$. In the setting described above, including Assumption~\eqref{eq:empirical_distr_zeroes_conv}, the empirical distribution $\nu\lsem f_{n}(\cdot\,;-\sigma^2/n)\rsem$ of zeroes of the solution $f_{n}(\theta; -\sigma^2/n)$ of the heat equation at time $t_{n} := -\sigma^2/n$, converges weakly (as $n=2d\to\infty$) to the free unitary convolution $\nu \boxtimes \fnorm_{\sigma^2}$ of $\nu$ and the free unitary  normal distribution $\fnorm_{\sigma^2}$.
\end{theorem}

We again rely on a result from finite free probability~\cite{marcus_spielman_srivastava,marcus}.
Recall from~\eqref{eq:finite_free_mult_conv_def} the definition of the finite free multiplicative convolution $\boxtimes_{n}$. The following fact proved in~\cite[Proposition~3.4]{arizmendi_garza_vargas_perales} (see also~\cite[Proposition~2.9]{kabluchko_rep_diff_free_poi} for the version stated here) states that,  in a suitable sense,  $\boxtimes_n$ converges to the free multiplicative convolution $\boxtimes$ as $n\to\infty$.
\begin{proposition}\label{prop:finite_free_mult_conv_to_free_mult}
Let $(p_n(z))_{n\in \N}$ and $(q_n(z))_{n\in \N}$ be sequences of polynomials in $\C[z]$ with $\deg p_n = \deg q_n = n$. Suppose that all roots of $p_n$ and $q_n$ are located on the unit circle and that,  as $n\to\infty$,  the empirical distributions of zeroes $\mu\lsem p_{n}\rsem$ and $\mu\lsem q_{n}\rsem$ converge weakly to two probability  measures $\nu$ and $\rho$ on $\bT$.
Then, all roots of the polynomial $p_n \boxtimes_n q_n$ are also located on $\bT$ and $\mu \lsem p_n \boxtimes_n q_n\rsem$ converges weakly to $\nu \boxtimes \rho$.
\end{proposition}

\begin{proof}[Proof of Theorem~\ref{theo:heat_flow_zeroes}.]
For $t\in \R$ we consider the algebraic polynomial $P_{2d}(z;t)$ in the complex variable $z$ defined by
\begin{equation}\label{eq:def_P_2d}
\frac{P_{2d}(\eee^{\ii \theta};t)}{\eee^{\ii d\theta}} = f_{2d}(\theta; t),
\end{equation}
More concretely, it follows from~\eqref{eq:heat_eq_solution} that
$$
P_{2d}(z;t)
=
\sum_{\ell = -d}^d  c_{\ell:2d} \eee^{-\frac t2 \ell^2}  z^{\ell+d}
=
\sum_{j = 0}^{2d}  c_{j-d:2d} \eee^{-\frac t2 (j-d)^2}  z^{j}
=
\sum_{j = 0}^{2d}  c_{j-d:2d} \eee^{-\frac t2 (j^2 - 2j d + d^2)}  z^{j}.
$$
It follows from~\eqref{eq:finite_free_mult_conv_def} and~\eqref{eq:hermite_poly_circ_def1} that for $t= t_{2d} =  -\sigma^2/(2d)<0$ we can write
\begin{align*}
P_{2d}(z;t)
&=
\eee^{-\frac {td^2} 2} \left(\sum_{j=0}^{2d}   (-1)^{2d-j} \binom {2d} j \eee^{ \frac {t}{2} j (2d-j)} z^j\right) \boxtimes_{2d} \left(\sum_{j = 0}^{2d}  c_{j-d:2d} z^j\right)\\
&=
\eee^{-\frac {td^2} 2} \left(H_{2d}\left(z; \frac{\sigma^2}{2d}\right) \boxtimes_{2d} P_{2d}(z;0)\right).
\end{align*}
Recall from Theorem~\ref{theo:weak_conv_zeroes_to_free_normal} that the empirical distribution of zeroes of $H_{2d}(z; \sigma^2/(2d))$ converges weakly to $\fnorm_{\sigma^2}$, while  the empirical distribution of zeroes of $P_{2d}(z;0)$ converges to $\nu$ by~\eqref{eq:def_P_2d} and~\eqref{eq:empirical_distr_zeroes_conv}:
$$
\frac {1}{2d} \sum_{\substack{z\in \C:\, H_{2d}(z; \sigma^2/(2d)) = 0}} \delta_z
\toweakd
\fnorm_{\sigma^2},
\qquad
\frac {1}{2d} \sum_{\substack{z\in \C:\, P_{2d}(z; 0) = 0}} \delta_z
=
\frac {1}{2d} \sum_{j=1}^{2d} \delta_{\eee^{\ii \theta_{j;2d}}}
\toweakd
\nu.
$$
To complete the proof, apply Proposition~\ref{prop:finite_free_mult_conv_to_free_mult}.
\end{proof}

\begin{figure}[t]
	\centering
	\includegraphics[width=0.49\columnwidth]{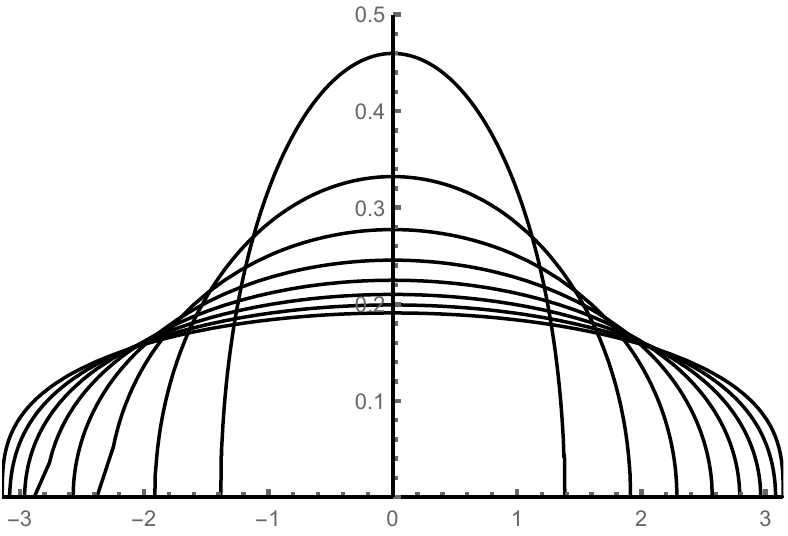}
	\includegraphics[width=0.49\columnwidth]{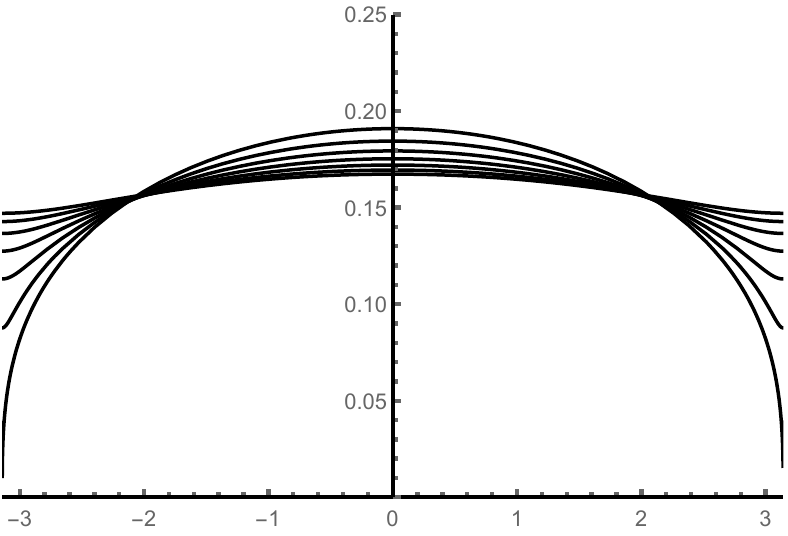}
	\caption{Densities of the free unitary normal distributions $\fnorm_{\sigma^2}$. The plots show the functions $\theta\mapsto f_{\sigma^2}(\eee^{\ii \theta})$ for $\theta \in [-\pi,\pi]$.  Left: $\sigma^2 = 0.5, 1, 1.5, \ldots, 4$. Right: $\sigma^2 = 4, 4.5,5, \ldots, 7$.}
\label{fig:densities}
\end{figure}

\subsection{Free unitary normal distribution and its properties}\label{sec:free_unitary_normal}
In this section we recall the notion of free multiplicative convolution $\boxtimes$ of probability measures on the unit circle which was introduced by Voiculescu in~\cite{voiculescu_symmetries} and~\cite{voiculescu_multiplication}; see also~\cite[\S~3.6]{voiculescu_nica_dykema_book}.
Also, we recall the definition of the free unitary normal distribution, which was introduced by Bercovici and Voiculescu~\cite{bercovici_voiculescu_levy_hincin} and studied in~\cite{biane,biane_segal_bargmann,zhong_free_brownian,zhong_free_normal,cebron}. Finally, we shall state some properties of this distribution.

Given two probability measures $\mu_1$ and $\mu_2$ on the unit circle $\bT$, it is possible to construct a $C^*$-probability space and two mutually free unitaries $u_1$ and $u_2$ with spectral distributions $\mu_1$ and $\mu_2$, respectively. Then, the free multiplicative convolution of $\mu_1$ and $\mu_2$, denoted by  $\mu_1 \boxtimes \mu_2$, is the spectral distribution of $u_1 u_2$. It does not depend on the choice of $u_1$ and $u_2$; see~\cite{voiculescu_nica_dykema_book} for more details.
Free multiplicative convolution is linearized by the $S$-transform which is defined as follows. The \emph{$\psi$-transform} of a probability measure $\mu$ on the unit circle $\bT$ is defined by
\begin{equation}\label{eq:psi_trafo_def}
\psi_{\mu}(z)
=
\int_\bT \frac{uz}{1-uz} \mu(\dint u) = \sum_{\ell=1}^\infty z^\ell \int_{\bT} u^\ell \mu(\dint u),
\qquad z\in \bD.
\end{equation}
Note that $\psi_\mu(z)$ is an analytic function on $\bD$ satisfying $\psi_\mu(0) = 0$ and $\Re \psi_\mu(z) > - \frac 12$. These properties characterize the class of $\psi$-transforms; see~\cite[Lemma~2.9]{franz_hasebe_schleissinger_monotone}.

If $\psi_\mu'(0) = \int_\bT u \mu(\dint u) \neq 0$, then the analytic function $\psi_\mu$ has an inverse on some sufficiently small disk around the origin and the \textit{$S$-transform}  of $\mu$ is defined by
\begin{equation}\label{eq:S_transf_def}
S_\mu(z) = \frac {1+z}{z} \psi^{-1}_\mu(z).
\end{equation}
The \emph{free unitary normal distribution} $\fnorm_{\sigma^2}$ with parameter $\sigma^2>0$
%and the \emph{free unitary Poisson distribution} $\Pi_t$  with parameter $t>0$,
introduced in~\cite[Lemmas~6.3]{bercovici_voiculescu_levy_hincin},
is a  probability measure on $\bT$ with
\begin{equation}\label{eq:S_transf_normal}
S_{\fnorm_{\sigma^2}}(z) = \eee^{\sigma^2 (z + \frac 12)}.
%\qquad
%S_{\Pi_t}(z) = \eee^{t/(z + \frac 12)}.
\end{equation}
The next theorem summarizes some properties of the free unitary normal distributions; see Figure~\ref{fig:densities} for the plots of their densities.
Almost all of these properties are known from the work of Biane~\cite{biane}, \cite[Proposition~10]{biane_segal_bargmann}; see also the discussion in~\cite[Proposition 2.24]{kemp} and~\cite[Remark~6.8]{levy_schur_weyl} for further pointers to the literature. For similar properties of the free unitary Poisson distribution we refer to~\cite[Section~5]{kabluchko_rep_diff_free_poi}.
\begin{theorem}\label{theo:properties_free_normal}
The density of the free unitary normal distribution $\fnorm_{\sigma^2}$ w.r.t.\ the length measure on $\bT$ is given by
\begin{equation}\label{eq:density_free_normal_zeta}
f_{\sigma^2}(\eee^{\ii \theta}) = \frac {2}{\pi \sigma^2} \Im \zeta_{\sigma^2/4}\left(\frac{\theta - \pi}{2}\right),
\qquad
\theta\in [-\pi, \pi),
\end{equation}
where $\zeta_{\sigma^2/4}(\cdot)$ is the function appearing in Theorem~\ref{theo:riemann_surface_z_i}.  The distribution $\fnorm_{\sigma^2}$ is invariant under complex conjugation meaning that $f_{\sigma^2}(\eee^{\ii \theta}) = f_{\sigma^2}(\eee^{-\ii \theta})$.
\begin{itemize}
\item[(a)] For $\sigma^2>4$, the function $\theta \mapsto f_{\sigma^2}(\eee^{\ii \theta})$ is strictly positive and real analytic on $\R$.
\item[(b)] For $0<\sigma^2<4$, the function $\theta \mapsto f_{\sigma^2}(\eee^{\ii \theta})$ is continuous on its period $[-\pi,\pi]$. It is strictly positive and real-analytic on the interval $(-m_{\sigma^2}, +m_{\sigma^2})$, and vanishes on its complement $[-\pi,\pi]\backslash (-m_{\sigma^2}, +m_{\sigma^2})$, where
    \begin{equation}\label{eq:support_free_normal}
    m_{\sigma^2} := 2 \arcsin \frac \sigma 2 + \frac \sigma 2 \sqrt{4- \sigma^2}= \int_0^{\sigma}\sqrt{4-t^2}\, \dint t.
    \end{equation}
At the points $\pm m_{\sigma^2}$, the function $\theta \mapsto f_{\sigma^2}(\eee^{\ii \theta})$  vanishes and has square-root singularities  with
\begin{equation}\label{eq:square_root_sing}
   f_{\sigma^2}(\eee^{\ii (m_{\sigma^2} - \eps)}) = f_{\sigma^2}(\eee^{\ii (\eps - m_{\sigma^2})})
   \sim
   \frac {1}{\pi \sigma^2}  \sqrt[4]{\frac{4\sigma^2}{4 - \sigma^2}} \cdot \sqrt \eps,
   \qquad
   \eps\downarrow 0.
    \end{equation}
\item[(c)] For $\sigma^2 = 4$, the function $\theta \mapsto f_{\sigma^2}(\eee^{\ii \theta})$ is strictly positive and real-analytic on the interval $(-\pi, \pi)$. At $\theta = \pm \pi$, it vanishes and has cubic-root singularities with
    \begin{equation}\label{eq:critical_cubic_root_sing}
    f_{\sigma^2}(\eee^{\ii (\pi -\eps)}) = f_{\sigma^2}(\eee^{\ii (\eps - \pi)}) \sim \frac{\sqrt 3}{4\pi} \left(\frac{3\eps}{2}\right)^{1/3},
    \qquad
    \eps\downarrow 0.
    \end{equation}
\end{itemize}
\end{theorem}

\begin{proof}[Proof of Theorem~\ref{theo:properties_free_normal}]
We shall show in Lemma~\ref{lem:psi_normal} below that
$$
\psi_{\fnorm_{\sigma^2}}(-\eee^{\ii \theta})
=
\frac {\ii \left(\theta - 2 \zeta_{\sigma^2/4}(\theta/2)\right)}{\sigma^2} -\frac 12
,
\qquad
\theta \in \bH.
$$
Consider the Poisson integral of $\fnorm_{\sigma^2}$ given for $z\in \bD$ by
$$
F_{\fnorm_{\sigma^2}}(z)
:=
\frac {1}{2\pi} \Re \int_\bT \frac{u + z}{u-z} \, \fnorm_{\sigma^2}(\dd u)
=
\frac {1}{2\pi} \Re \int_\bT \frac{1+ z\bar u}{1 - z \bar u} \, \fnorm_{\sigma^2}(\dd u)
=
\frac 1 {2\pi} \Re (1 + 2 \psi_{\fnorm_{\sigma^2}}(\bar z))
.
$$
Since the function $\zeta_{\sigma^2/4}$ admits a continuous extension to $\overline \bH$, the above formula defines $F_{\fnorm_{\sigma^2}}(z)$ as a continuous function on $\overline \bD$. Under these circumstances, it is well known (see, e.g., \cite[Lemma~5.2]{kabluchko_rep_diff_free_poi}) that the density of $\fnorm_{\sigma^2}$ with respect to the length measure on $\bT$ is given by
\begin{align*}
f_{\sigma^2}(\eee^{\ii \theta})
&=
F_{\fnorm_{\sigma^2}}(\eee^{\ii \theta}) =
\frac 1 {2\pi} \Re (1 + 2 \psi_{\fnorm_{\sigma^2}}(\eee^{-\ii \theta}))
=
\frac 1 {2\pi} \Re (1 + 2 \psi_{\fnorm_{\sigma^2}}( - \eee^{\ii(\pi -\theta})))\\
&=
\frac {2}{\pi \sigma^2} \Im \zeta_{\sigma^2/4}\left(\frac{\theta - \pi}{2}\right)
\end{align*}
for all $\theta\in \R$. All other claims follow from the properties of the function $\zeta_{\sigma^2/4}$ derived in~\cite[Section~4]{kabluchko_rep_diff_free_poi}. In particular, Equation~\eqref{eq:square_root_sing} follows  from~\eqref{eq:density_free_normal_zeta} and~\cite[Remark~4.3]{kabluchko_rep_diff_free_poi}, while Equation~\eqref{eq:critical_cubic_root_sing} is a consequence of~\cite[Lemma~4.5]{kabluchko_rep_diff_free_poi}.
\end{proof}

\begin{remark}
It follows from~\eqref{eq:density_free_normal_zeta} and the series expansion of $\zeta_{\sigma^2/4}$ given in~\cite[Theorem~4.11]{kabluchko_rep_diff_free_poi} that
\begin{equation}\label{eq:r_t_taylor_series}
f_{\sigma^2}(\eee^{\ii \theta})
=
\frac 1 {2\pi} + \frac{1}{\pi } \sum_{\ell=1}^\infty \frac{1}{\ell} \eee^{-\ell \sigma^2/2 } q_{\ell-1} (-\ell \sigma^2) \cos(\ell \theta),
%\qquad \theta
%\in [-\pi,\pi],
\end{equation}
where  $q_m (x) =  \sum_{j=0}^{m} \frac{x^j}{j!} \binom {m+1}{j+1}$, $m\in \N_0$, and the series converges uniformly in $\theta\in \R$. Note that~\eqref{eq:r_t_taylor_series} characterizes the moments of $\fnorm_{\sigma^2}$. In this form, the result be found in~\cite[Lemma~1]{biane}.
\end{remark}

\begin{remark}
%A function $\Omega(u)$ very similar to $m_{\sigma^2}$ appears as the limit shape of random Young diagrams discovered by Vershik and Kerov~\cite{vershik_kerov} and Logan and Shepp~\cite{logan_shepp}; see also~\cite[Theorems~1.20 and~1.22]{romik_book}.
The function $\zeta_t$ is closely related to the function $\chi(t,z)$ appearing in~\cite[p.~273]{biane_segal_bargmann}.  A function very similar to $m_{\sigma^2}$ appears in~\cite[Theorem~5.1]{romik_sniady}.
\end{remark}
It follows from~\eqref{eq:r_t_taylor_series} that, as $\sigma^2\to \infty$, the density $f_{\sigma^2}(\eee^{\ii \theta})$ converges to $1/(2\pi)$ uniformly, which implies that the free unitary normal distribution converges to the uniform distribution weakly on $\bT$. On the other hand, the next result states that, as $\sigma^2\to 0$, the free unitary normal  distribution $\fnorm_{\sigma^2}$ can be approximated by a Wigner semicircle distribution supported on a small arc of length $\sim 4 \sigma$ centered at $1$.
\begin{proposition}\label{prop:free_unitary_normal_conv_wigner}
For all $x\in (-2,2)$ we have $\lim_{\sigma^2\downarrow 0} \sigma f_{\sigma^2} (\eee^{\ii \sigma x}) = \frac 1 {2 \pi} \sqrt{4 - x^2}$.
\end{proposition}
\begin{proof}
We shall only sketch the idea without giving full details. First of all, from~\eqref{eq:support_free_normal} we have $m_{\sigma^2} \sim 2 \sigma$ as $\sigma^2\downarrow 0$. Let us take $x=0$. It has been argued in~\cite[Section~4]{kabluchko_rep_diff_free_poi} that for all $\tau\geq 0$, we have $\zeta_{\sigma^2/4}(-\frac \pi 2 + \ii \tau) = -\frac \pi 2 + \ii \tilde y_{\sigma^2/4}(\tau)$, where $\tilde y= \tilde y_{\sigma^2/4}(\tau)>0$ is the unique positive solution of $\tilde y - \frac 14 \sigma^2 \cotanh \tilde y  = \tau$. For $\tau = 0$ the equation takes the form $\tilde y = \frac 14 \sigma^2 \cotanh \tilde y$ and its unique positive solution satisfies $\tilde y_{\sigma^2/4}(0) \sim \frac \sigma 2$ as $\sigma^2 \downarrow 0$. It follows that $\zeta_{\sigma^2/4}(-\frac \pi 2) = -\frac \pi 2 + \ii \frac \sigma2 + o(\sigma)$, and it follows from~\eqref{eq:density_free_normal_zeta} that $\sigma f_{\sigma^2}(1) \to \frac 1 \pi$ as $\sigma^2 \downarrow 0$. For general $x\in  (-2,2)$, Equation~\eqref{eq:density_free_normal_zeta} expresses $f_{\sigma^2} (\eee^{\ii \sigma x})$ through the value $\zeta = \zeta_{\sigma^2/4}(\frac{\sigma x - \pi}{2})$ which solves the equation
$$
\zeta - \frac{\sigma^2}{4} \tan \zeta = \frac{\sigma x - \pi}{2}.
$$
As $\sigma^2\downarrow 0$, we can search for a solution in the form $\zeta = -\frac \pi 2 + c(x) \sigma + o(\sigma)$ with some unknown $c(x)$ satisfying $c(0) = \frac \ii 2$ by the above analysis of the case $x=0$. Inserting this into the equation for $\zeta$ and taking $\sigma^2\downarrow 0$,  we obtain $4c(x) + \frac 1 {c(x)} = 2x$. The solution is $c(x) = \frac 14 x + \frac 14 \sqrt{x^2 - 4}$. It follows that $\Im \zeta_{\sigma^2/4}(\frac{\sigma x - \pi}{2}) \sim \frac \sigma 4 \sqrt{4 - x^2}$ for $x\in  (-2,2)$, as $\sigma^2\downarrow 0$. Inserting this into~\eqref{eq:density_free_normal_zeta} completes the argument.
\end{proof}

\section{Proofs of Theorems~\ref{theo:weak_conv_zeroes_to_free_normal} and~\ref{theo:hermite_unitary_asympt}} \label{sec:proofs}
In this section we prove Theorems~\ref{theo:weak_conv_zeroes_to_free_normal} and~\ref{theo:hermite_unitary_asympt}. Throughout the proof, we use the following notational convention: $\theta$ denotes a variable ranging in the upper half-plane $\bH$, while $z:= -\eee^{\ii \theta}$ is a variable ranging in the unit disk $\bD$. Most of the time we shall be occupied with the proof of the following result.
\begin{proposition}\label{prop:log_der_limit_small_r}
Fix $\sigma^2>0$. If $r = r(\sigma^2) > 0$ is sufficiently small, then locally uniformly in $|z|\leq r$ we have
\begin{equation}\label{eq:H_n_log_der_limit_proof_1}
\lim_{n\to\infty} \frac 1n \cdot \frac{H_n'(z;\sigma^2/n)}{H_n(z; \sigma^2/n)}
=
\frac{-\eee^{-\ii \theta}}{1+ \eee^{-2 \ii \zeta_{\sigma^2/4}(\theta/2)}}
=
-\eee^{-\ii \theta}\left(\frac \ii 2 \tan \zeta_{\sigma^2/4} (\theta/2) + \frac 12\right).
\end{equation}
\end{proposition}

As we argued in Remark~\ref{rem:analytic_functions_well_defined}, the right-hand side of~\eqref{eq:H_n_log_der_limit_proof_1} can be considered as an analytic function of $z\in \bD$ including the value $z=0$ (corresponding to $\Im \theta \to +\infty$) where it equals $-\eee^{-\sigma^2/2}$.

\subsection{Proof of Theorems~\ref{theo:weak_conv_zeroes_to_free_normal} and~\ref{theo:hermite_unitary_asympt} assuming Proposition~\ref{prop:log_der_limit_small_r}}
First we need to compute the $\psi$-transforms of the empirical distribution of zeroes of $H_n(z;\sigma^2/n)$ and of the free unitary normal distribution $\fnorm_{\sigma^2}$. This is done in the next two lemmas.
%Recall from~\eqref{eq:Psi_mu_n_sigma^2} that $\mu_{n;\sigma^2} = $.
\begin{lemma}\label{lem:psi_mu_n_sigma^2}
The $\psi$-transform of the probability measure $\mu_{n;\sigma^2}:=\frac 1n \sum_{z\in \bT:\, H_n(z; \sigma^2/n) = 0} \delta_z$ is given by
$$
\psi_{\mu_{n;\sigma^2}}(z) = -\frac z n \cdot \frac{H_n'(z; \sigma^2/n)}{H_n(z;\sigma^2/n)},
\qquad z\in \bD.
$$
\end{lemma}
\begin{proof}
Using the definition of the $\psi$-transform given in~\eqref{eq:psi_trafo_def} and the identity $\bar u = 1/u$ for $u\in \bT$, we can write
\begin{align*}
\overline{\psi_{\mu_{n; \sigma^2}}(\bar z)}
&=
\frac 1 n \sum_{\ell=1}^\infty z^\ell  \sum_{\substack{u\in \bT:\\ H_n(u; \sigma^2/n) = 0}} \frac 1 {u^\ell}
=
\frac 1 n \sum_{\substack{u\in \bT:\\ H_n(u; \sigma^2/n) = 0}} \sum_{\ell=1}^\infty \frac {z^\ell} {u^\ell}
=
\frac 1 n \sum_{\substack{u\in \bT:\\ H_n(u; \sigma^2/n) = 0}} \frac{z}{u-z}
\\
&=
-\frac z n \cdot \frac{H_n'(z; \sigma^2/n)}{H_n(z;\sigma^2/n)}.
\end{align*}
To complete the proof, note that the polynomials $H_n(z; \sigma^2/n)$ have real coefficients, implying that its non-real zeroes come in complex-conjugated pairs and hence $\overline{\psi_{\mu_{n; \sigma^2}}(\bar z)} = \psi_{\mu_{n; \sigma^2}}(z)$.
\end{proof}
%The next lemma expresses the $\psi$-transform of  through the function $\zeta_{\sigma^2/4}$.
\begin{lemma}\label{lem:psi_normal}
The $\psi$-transform of the free unitary normal distribution $\fnorm_{\sigma^2}$ with parameter $\sigma^2>0$ is given by
$$
\psi_{\fnorm_{\sigma^2}}(-\eee^{\ii \theta})
=
- \frac \ii 2 \tan \zeta_{\sigma^2/4}(\theta/2)- \frac 12
=
\frac {\ii \left(\theta - 2 \zeta_{\sigma^2/4}(\theta/2)\right)}{\sigma^2} -\frac 12
,
\qquad
\theta \in \bH.
$$
\end{lemma}
\begin{proof}
By the definition of $\fnorm_{\sigma^2}$ given in~\eqref{eq:S_transf_normal}, for all complex $w$ with sufficiently small $|w|$ we have
\begin{equation}\label{eq:psi_z_inverse_proof_free_normal}
w = \psi_{\fnorm_{\sigma^2}}\left(\frac w {1+w} \eee^{\sigma^2 (w + \frac 12)}\right).
\end{equation}
Let us use the shorthand $\zeta:= \zeta_{\sigma^2/4}(\theta/2)$ with some $\theta \in \bH$ and put
$$
w:=  - \frac \ii 2 \tan \zeta - \frac 12 = \frac {\ii(\theta - 2 \zeta)}{\sigma^2}  - \frac 12.
$$
If $\Im \theta$ is sufficiently large, then $|w|$ is sufficiently small; see~\eqref{eq:zeta_limit_behavior}.  Also, the definition of $w$ implies that
$$
\frac {w}{1+w} \eee^{\sigma^2 (w+ \frac 12 )}
=
\frac { - \frac \ii 2 \tan \zeta  - \frac 12}{ - \frac \ii 2 \tan \zeta  + \frac 12 } \eee^{\ii (\theta - 2 \zeta)}
=
\frac { - \frac{\eee^{\ii \zeta} -\eee^{-\ii \zeta}}{\eee^{\ii \zeta} + \eee^{-\ii \zeta}}  - 1}{ - \frac{\eee^{\ii \zeta} -\eee^{-\ii \zeta}}{\eee^{\ii \zeta} + \eee^{-\ii \zeta}}  + 1 } \eee^{\ii (\theta - 2 \zeta)}
=
- \eee^{2\ii \zeta}\eee^{\ii (\theta - 2 \zeta)}
=
-\eee^{\ii \theta}.
$$
Inserting this into~\eqref{eq:psi_z_inverse_proof_free_normal} yields $\psi_{\fnorm_{\sigma^2}}(-\eee^{\ii \theta}) = w = - \frac \ii 2 \tan \zeta_{\sigma^2/4}(\theta/2) - \frac 12$ for all $\theta \in \bH$ with sufficiently large $\Im \theta$. The latter restriction can be dropped by uniqueness of analytic continuation since both sides are analytic functions of $\theta\in \bH$.
\end{proof}

%Let us now prove  Theorem~\ref{theo:weak_conv_zeroes_to_free_normal}.
\begin{proof}[Proof of Theorems~\ref{theo:weak_conv_zeroes_to_free_normal} and~\ref{theo:hermite_unitary_asympt} assuming Proposition~\ref{prop:log_der_limit_small_r}.]
From Proposition~\ref{prop:log_der_limit_small_r} combined with Lemmas~\ref{lem:psi_mu_n_sigma^2} and~\ref{lem:psi_normal} we conclude that
\begin{equation}\label{eq:psi_transforms_converge}
\psi_{\mu_{n;\sigma^2}}(z)
=
-
\frac z n \cdot \frac{H_n'(z; \sigma^2/n)}{H_n(z;\sigma^2/n)}
\ton
%- z \frac{\partial_t S(t_0;z)}{S(t_0;z)}
-\frac{1}{1+ \eee^{-2 \ii \zeta_{\sigma^2/4}(\theta/2)}}
=
-\frac \ii 2 \tan \zeta_{\sigma^2/4} (\theta/2) - \frac 12
=
\psi_{\fnorm_{\sigma^2}}(z),
\end{equation}
locally uniformly on $\{|z|\leq r\}$.  By standard results, see Lemmas 2.6 and 2.11 in~\cite{franz_hasebe_schleissinger_monotone}, this already implies  uniform convergence on  $\{|z| \leq  R\}$ for every $R\in (0,1)$ and the weak convergence $\mu_{n;\sigma^2} \to \fnorm_{\sigma^2}$. For convenience of the reader, we provide a full proof.

By Cauchy's integral formula, \eqref{eq:psi_transforms_converge} implies the convergence of the derivatives of any order at $0$ of the above $\psi$-transforms. By~\eqref{eq:psi_trafo_def} this means that the Fourier coefficients of $\mu_{n;\sigma^2}$ converge to those of $\fnorm_{\sigma^2}$, namely
$$
\int_{\bT} u^\ell \mu_{n;\sigma^2}(\dint u) =  \frac{\psi_{\mu_{n;\sigma^2}}^{(\ell)}(0)}{\ell!}  \ton \frac{\psi_{\fnorm_{\sigma^2}}^{(\ell)}(0)}{\ell!} = \int_{\bT} u^\ell \fnorm_{\sigma^2}(\dint u),
$$
for all $\ell\in  \N$. Since we are dealing with measures invariant under complex conjugation, this convergence continues to hold for all $\ell\in \Z$.
By~\cite[p.~50]{billingsley_book68}, it  follows that $\mu_{n;\sigma^2} \to \fnorm_{\sigma^2}$ weakly on $\bT$, thus proving Theorem~\ref{theo:weak_conv_zeroes_to_free_normal}.

To prove Theorem~\ref{theo:hermite_unitary_asympt}, we need the following classical consequence of Montel's fundamental normality test; see~\cite[p.~252]{burckel_book} and also~\cite[p.~219]{burckel_book} for a slightly weaker Vitali--Porter theorem.
%%%There is a weaker form requiring local boundedness of the family - Vitali-Porter theorem. It is also sufficient for our purposes: just map the half-plane to the disc.
\begin{theorem}[Carath\'eodory--Landau theorem]\label{theo:normality_conv}
Let  $f_1,f_2,\ldots$ and $f$ be analytic functions defined on the unit disk $\bD$  and taking values in $\bC \bsl \{a_1, a_2\}$, where $a_1, a_2\in \C$ are distinct.  If $\lim_{n\to\infty} f_n(z_i) = f(z_i)$ for all $i\in \N$, where $z_1,z_2,\ldots \in \bD$  is a sequence of distinct points having an accumulation point in $\bD$, then $f_n\to f$ locally uniformly on $\bD$.
\end{theorem}

We are now in position to  prove~\eqref{eq:asympt_H_n_log_der} of Theorem~\ref{theo:hermite_unitary_asympt}. Since $\Re \psi_{\mu_{n;\sigma^2}}(z) >  - \frac 12$ for all $z\in \bD$,  Theorem~\ref{theo:normality_conv} lifts~\eqref{eq:psi_transforms_converge} to locally uniform convergence on $\bD$. That is, for every $R\in (0,1)$,  we have $\lim_{n\to\infty} \psi_{\mu_{n;\sigma^2}}(z)  = \psi_{\fnorm_{\sigma^2}}(z)$ uniformly on $\{|z|\leq R\}$. In view of Lemmas~\ref{lem:psi_mu_n_sigma^2} and~\ref{lem:psi_normal} we conclude that
\begin{equation}\label{eq:lim_log_der_conv_proof}
\lim_{n\to\infty} \frac 1n \cdot \frac{H_n'(z;\sigma^2/n)}{H_n(z; \sigma^2/n)}
=
\frac{-\eee^{-\ii \theta}}{1+ \eee^{-2 \ii \zeta_{\sigma^2/4}(\theta/2)}}
=
-\eee^{-\ii \theta} \left(\frac \ii 2 \tan \zeta_{\sigma^2/4} (\theta/2) + \frac 12\right)
\end{equation}
uniformly on $\{|z|\leq R\}$. This proves~\eqref{eq:asympt_H_n_log_der}. %of Theorem~\ref{theo:hermite_unitary_asympt}.

Let us prove~\eqref{eq:log_asympt_H_n} of Theorem~\ref{theo:hermite_unitary_asympt}. Denote the right-hand side of~\eqref{eq:lim_log_der_conv_proof} by $h_{\sigma^2/4}(z)$ with the usual convention $z= -\eee^{\ii \theta}$. Integrating the uniform convergence in~\eqref{eq:lim_log_der_conv_proof} along the segment joining $0$ and $z$ we obtain
\begin{multline}\label{eq:proof_tech_log_H_n}
\frac 1n \log \frac{H_n(z;\sigma^2/n)}{(-1)^n}
=
\int_0^z \frac 1n \cdot \frac{H_n'(y;\sigma^2/n)}{H_n(y; \sigma^2/n)} \dint y
\\
\ton
\int_0^z  h_{\sigma^2/4}(y) \dint y
%\int_0^z \frac{-\eee^{-\ii \theta} \dint y}{1+ \eee^{-2 \ii \zeta_{\sigma^2/4}(\theta/2)}}
=
\log \left(1+\eee^{2\ii \zeta_{\sigma^2/4}(\theta/2)}\right) - \frac{\sigma^2}{2\left(1+ \eee^{-2 \ii \zeta_{\sigma^2/4}(\theta/2)}\right)^2}.
\end{multline}
To prove the last equality, one observes that the right-hand side of~\eqref{eq:proof_tech_log_H_n} vanishes at $z=0$ and that its derivative in $z$ is $h_{\sigma^2/4}(z)$; see~\eqref{eq:S_der} for the details of the calculation.
\end{proof}

\subsection{Proof of Proposition~\ref{prop:log_der_limit_small_r}}
%The rest of the paper is devoted to the proof of Proposition~\ref{prop:log_der_limit_small_r}.
We start with an integral representation of the unitary Hermite polynomials. A related result can be found in~\cite[Proposition~2.1]{shamis_zeitouni}.
%With a slightly different notation, the next lemma can be found in~\cite[Proposition~2.1]{shamis_zeitouni}.
\begin{lemma}\label{lem:hubbard_stratonovich}
For all $n\in \N$, $\sigma^2>0$ and $z\in \C$ we have
\begin{equation}\label{eq:H_n_as_int}
H_n(z;\sigma^2/n) = (-1)^n  \frac{\sqrt n}{\sqrt{2\pi \sigma^2}} \int_{-\infty}^{+\infty} (1 - z\eee^t)^n \eee^{-\frac n2  \left(\frac{t}{\sigma} + \frac \sigma 2\right)^2}  \dint t.
\end{equation}
\end{lemma}
\begin{proof}
As in~\cite{shamis_zeitouni}, we use the Hubbard-Stratonovich trick, that is the identity
$\eee^{\frac 12 a^2} = \frac 1 {\sqrt {2\pi}} \int_{-\infty}^{+\infty} \eee^{- \frac 1 2 s^2} \eee^{a s}\dint s$, for all $a\in \C$.  With $a = \sigma (j-\frac n2)/\sqrt n$ it follows  that
$$
\eee^{\frac 1{2n} \sigma^2 (j- \frac n2)^2}
=
\frac 1 {\sqrt {2\pi}} \int_{-\infty}^{+\infty} \eee^{- \frac 1 2 s^2} \eee^{s \sigma(j-\frac n2)/\sqrt n}\dint s.
$$
With this in mind, we can represent the unitary Hermite polynomials as follows:
\begin{align*}
H_n(z;\sigma^2/n)
&=
\sum_{j=0}^n   \binom nj  (-1)^{n-j}   z^j \exp\left\{\frac {\sigma^2 (j^2 - jn)}{2n}\right\} \\
&=
\eee^{-\frac 18 \sigma^2n} \frac 1 {\sqrt {2\pi}}  \int_{-\infty}^{+\infty} \sum_{j=0}^n \binom nj (-1)^{n-j} z^j  \eee^{-\frac 12 s^2} \eee^{s \sigma (j-\frac n2)/\sqrt n}\dint s\\
&=
(-1)^n \eee^{-\frac 18 \sigma^2n } \frac {1} {\sqrt {2\pi}}  \int_{-\infty}^{+\infty} (1 - z\eee^{s \sigma/\sqrt n})^n  \eee^{-\frac 12 s^2} \eee^{ - \frac 12 s \sigma \sqrt n}\dint s\\
&=
(-1)^n \frac {1} {\sqrt {2\pi}}  \int_{-\infty}^{+\infty} (1 - z\eee^{s \sigma/\sqrt n})^n  \eee^{-\frac 12(s + \frac12 \sigma \sqrt n)^2} \dint s.
\end{align*}
After the substitution $t=s\sigma/\sqrt n$ we arrive at~\eqref{eq:H_n_as_int}.
\end{proof}

Our next aim is to prove the following result which will imply Proposition~\ref{prop:log_der_limit_small_r} by differentiation.

\begin{proposition}\label{prop:log_asympt_H_n_small_z}
Fix some $\sigma^2>0$. If $r=r(\sigma^2)>0$ is sufficiently small, then uniformly over the disk $\{z\in \C:\, |z|\leq r\}$ we have
$$
\lim_{n\to\infty}
\frac 1n \log \frac {H_n(z;\sigma^2/n)}{(-1)^n}
%\Bigg(\frac{\sqrt n}{\sqrt{2\pi \sigma^2}} \int_{-\infty}^{+\infty} (1 - z\eee^t)^n \eee^{-\frac n2  \left(\frac{t}{\sigma} + \frac \sigma 2\right)^2}  \dint t\Bigg)
=
\log (1+\eee^{2\ii \zeta_{\sigma^2/4}(\theta/2)}) - \frac{\sigma^2}{2(1+ \eee^{-2 \ii \zeta_{\sigma^2/4}(\theta/2)})^2},
%,
%\text{ as }
%n \to\infty.
$$
where $\theta\in \bH$ is such that $-\eee^{\ii \theta} = z$ if $z\neq 0$. In the case $z=0$, which corresponds to $\Im \theta \to +\infty$, the right-hand side is defined to be $0$, by continuity. The branch of the logarithm on the left-hand side is chosen  such that $\log 1=0$ at $z=0$ and  the function $\log(\ldots)$ is continuous (and analytic).
\end{proposition}

We shall write the integral in~\eqref{eq:H_n_as_int} as $\int_{-\infty}^0 + \int_{0}^{+\infty}$ and analyze both summands separately. As we shall show, the main contribution comes from the negative half-axis.

\medskip
\noindent
\textit{Saddle point analysis.} With the help of the saddle-point method~\cite[\S~45.4, p.~423]{sidorov_fedoryuk_shabunin_book} we shall analyze the integral $\int_{-\infty}^0 \eee^{n S(t;z)}  \dint t$, where
\begin{equation}\label{eq:S_def}
S(t;z) = \log (1-z\eee^t) - \frac 12  \left(\frac{t}{\sigma} + \frac \sigma 2\right)^2,
\qquad
|z|<1, \; \Re t < 0.
\end{equation}
Observe that $S(t;z)$ is an analytic function of its two variables since for $|z|< 1$ and $\Re t < 0$ we have $|z\eee^t| < 1$, implying  that $\log (1-z \eee^t)$ is well-defined and analytic.
%We claim that if $r = r(\sigma^2)>0$ is sufficiently small, then uniformly in  $|z|\leq r$ we have
%$$
%\lim_{n\to\infty}
%\frac 1n \log  \int_{-\infty}^0 \eee^{n S(t;z)}  \dint t
%=
%\log (1+\eee^{2\ii \zeta_{\sigma^2/4}(\theta/2)}) - \frac{\sigma^2}{2(1+ \eee^{-2 \ii \zeta_{\sigma^2/4}(\theta/2)})^2}.
%%\text{ as }
%%n \to\infty.
%$$
The saddle-point equation takes the form
\begin{equation}\label{eq:saddle_point}
\frac{\dint}{\dint t} S(t; z) = - \frac{z\eee^t}{1 - z \eee^{t}} - \frac {t}{\sigma^2} - \frac 12 = 0.
\end{equation}
\begin{lemma}\label{lem:crit_point_explicit}
For every $z\in \bD$, Equation~\eqref{eq:saddle_point} has a unique solution $t_0 := t_0(z; \sigma^2)$ in the left half-plane $\{\Re t < 0\}$. For $z \in \bD\backslash\{0\}$, the solution is given by
\begin{equation}\label{eq:saddle_point_formula}
t_0(z; \sigma^2) = \frac {\ii \sigma^2}{2}  \tan \zeta_{\sigma^2/4}(\theta/2) = 2\ii \left(\zeta_{\sigma^2/4}(\theta/2) - \theta/2\right),
\end{equation}
where we recall the notation $z= -\eee^{\ii \theta}$ for some $\theta \in \bH$.
For $z=0$, the solution is $t_0(0;\sigma^2) = -\frac 12 \sigma^2$.
\end{lemma}
\begin{proof}
Let $z=-\eee^{\ii \theta}\neq 0$ with $\theta \in \bH$.
Every complex number $t$ with $\Re t <0$ can be represented as $t = \frac {\ii \sigma^2}{2} \tan \zeta$ with some $\zeta\in \bH$ which is unique up to an additive term of the form $\pi n$, $n\in \Z$.   With this notation, Equation~\eqref{eq:saddle_point} takes the form
$$
0
=
\frac 12 + \frac {t}{\sigma^2} + \frac{z\eee^t}{1 - z \eee^{t}}
=
\frac 12 + \frac \ii 2 \tan \zeta - \frac{1}{1- z^{-1}\eee^{-t}}
=
\frac{1}{1 + \eee^{-2\ii \zeta}} - \frac{1}{1 +  \eee^{-\ii \theta - \frac {\ii \sigma^2}{2}\tan \zeta}}.
$$
It follows that $2\ii \zeta = \ii \theta + \frac {\ii \sigma^2}{2} \tan \zeta + 2 \pi \ii n$ for some $n\in \Z$. Hence, for $\zeta^*:=\zeta - \pi n$ we obtain the equation $\zeta^* = \frac \theta 2 + \frac{\sigma^2}{4} \tan \zeta^*$. Since $\zeta^*\in \bH$, it follows that $\zeta^* = \zeta_{\sigma^2/4}(\theta/2)$, see Theorem~\ref{theo:riemann_surface_z_i}, and the proof is complete.
\end{proof}

\begin{figure}[t]
\centering
\includegraphics[width=0.45\columnwidth]{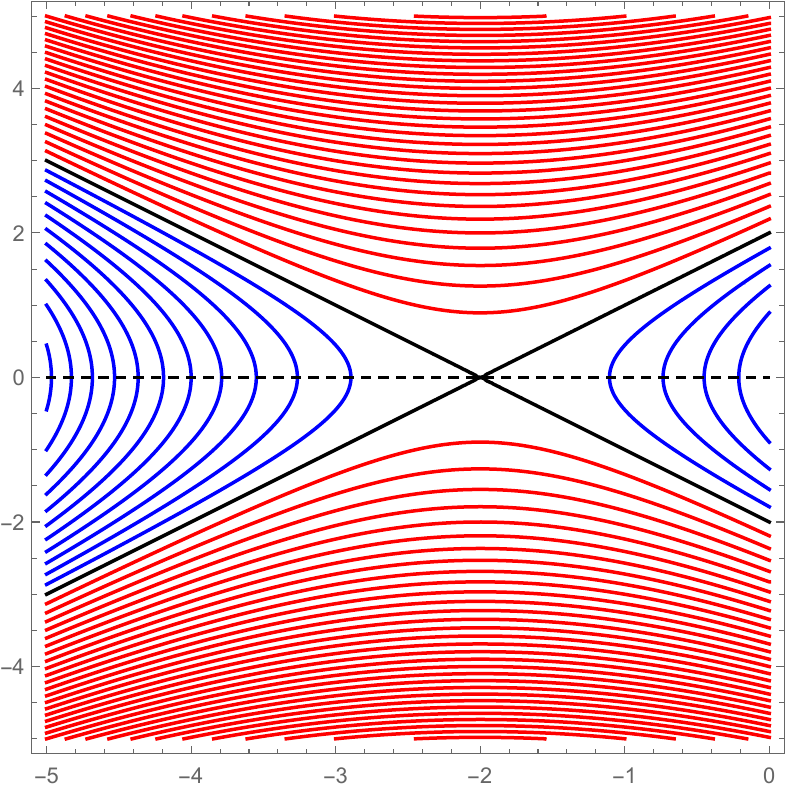}
\includegraphics[width=0.45\columnwidth]{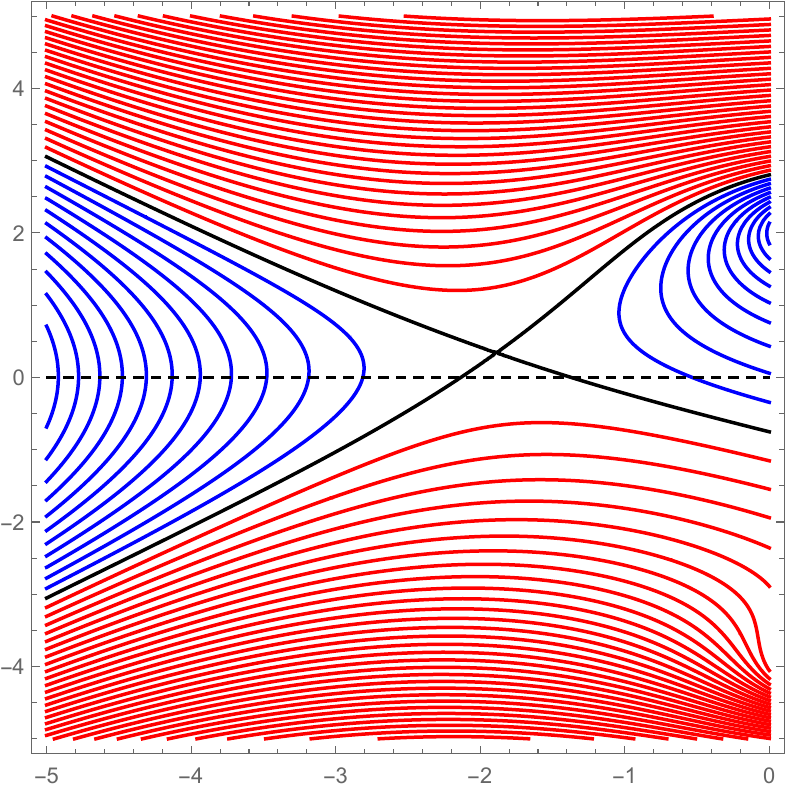}
\caption{Level lines of the function $t \mapsto \Re S(t;z)$ in the left half-plane $\{\Re t \leq 0\}$. Left panel: $z=0$.  Right panel: $z\neq 0$.  In both cases, $\sigma=2$. The level set passing through the saddle point is shown in black.  Level lines where the function takes smaller (respectively, larger) values than at the saddle point are shown in blue (respectively, red).  }
\label{fig:level_lines}
\end{figure}

We are now going to apply the saddle-point method~\cite[\S~45.4, p.~423]{sidorov_fedoryuk_shabunin_book} to the integral $\int_{-\infty}^0 \eee^{n S(t;z)}  \dint t$. To this end, we shall replace the initial contour of  integration by a contour $\gamma$ (depending on $z$ and $\sigma^2$) with the following properties. The contour $\gamma$ starts at $-\infty$, ends at $0$, and stays in the left half-plane $\Re t\leq 0$. Furthermore, $\gamma$ passes through the saddle point $t_0= t_0(z)=t_0(z;\sigma^2)$,  and satisfies $\Re S(t;z) < \Re S(t_0(z); z)$ for all $t\in \gamma\backslash\{t_0(z)\}$. Finally,  in the neighborhood of the point $t_0(z)$, the contour is required to pass through both sectors in which $\Re S(t;z) < \Re S(t_0(z);z)$.

We shall prove the existence of such a contour provided $|z|$ is sufficiently small (even though we conjecture that it exists for all $z\in \bD$). Let $\partial_1 S$ and $\partial_2 S$ denote the partial derivatives of $S(\cdot, \cdot)$ in  the first/second argument.
First of all, for $z=0$, the function $S(t;0) = - \frac 12 (\frac{t}{\sigma} + \frac \sigma 2)^2$ attains its strict maximum on $(-\infty, 0]$ at $t_0(0) = -\frac 12 \sigma^2$. This critical point is simple meaning that $(\partial_1^2 S)(-\frac 12 \sigma^2;0) < 0$.   It follows that the negative half-line $(-\infty, 0]$ satisfies the conditions listed above; see the left panel of Figure~\ref{fig:level_lines}. Observe also that in a small disk $D$ centered at $-\frac 12 \sigma^2$, the set where $\Re S(t;0) < \Re S(-\frac 12 \sigma^2;0)$ consists of two sectors whose boundaries are straight lines crossing $D$ at four points $A_1, A_2, A_3, A_4$. After a small perturbation of $z$, we still have a critical point at $t_0(z)$ which is close to $t_0(0)$ and simple, by analyticity of $t_0(z)$ and $S$; see Lemma~\ref{lem:crit_point_explicit}.  It follows that, in the disk $D$, the set where $\Re S(t;z) < \Re S(t_0(z);z)$ still consists of two sectors; see the right panel of Figure~\ref{fig:level_lines}.  The four points where the boundaries of these sectors cross $D$ are close to the points $A_1,A_2,A_3,A_4$, again by analyticity. It follows that the saddle-point contour $\gamma$ satisfying the above condition exists (and can be obtained by perturbing the segment of the real line contained in $D$).

Applying the saddle-point asymptotics~\cite[\S~45.4, p.~423]{sidorov_fedoryuk_shabunin_book},  we have
\begin{equation}\label{eq:exact_asymptotics}
\int_{-\infty}^0 \eee^{n S(t;z)}  \dint t
\sim
\sqrt{-\frac{2\pi}{n (\partial_1^2 S)(t_0(z); z)}} \eee^{n S(t_0(z);z)},
\qquad
\text{ as }
n\to\infty.
\end{equation}
This holds pointwise in $z$ provided $|z|$ is sufficiently small (the uniformity in $z$ will be addressed later). A formula for $(\partial_1^2 S)(t_0(z); z)$ will be given in~\eqref{eq:second_der_S}. The choice of the square root branch, irrelevant for our purposes, is explained in~\cite[p.~425]{sidorov_fedoryuk_shabunin_book}.  It follows from~\eqref{eq:S_def} and~\eqref{eq:saddle_point_formula} that for $z = -\eee^{\ii \theta} \neq 0$ we have
\begin{equation}\label{eq:S_explicit}
S(t_0(z); z)
=
\log (1-z\eee^{t_0(z)}) - \frac 12 \left(\frac {t_0(z)}{\sigma} + \frac \sigma2\right)^2
%&=
%\log(1 + \eee^{2\ii \zeta_{\sigma^2/4}(\theta/2)}) - \frac {\sigma^2}8 \left(\ii \tan \zeta_{\sigma^2/4}(\theta/2) + 1\right)^2 + \frac {\sigma^2}{8}\\
=
 \log (1+\eee^{2\ii \zeta_{\sigma^2/4}(\theta/2)}) - \frac{\sigma^2}{2(1+ \eee^{-2 \ii \zeta_{\sigma^2/4}(\theta/2)})^2}.
\end{equation}
Note that for $z=0$ we have $S(t_0(0);0) = 0$. For future use let us note the identity
\begin{align}
\frac{\dint}{\dint z} S(t_0(z);z)
&=
(\partial_1 S)(t_0(z);z) \partial_z t_0(z) + (\partial_2 S) (t_0(z); z)
=
(\partial_2 S) (t_0(z); z)\notag\\
&=
-\frac{\eee^{t_0(z)}}{1 - z \eee^{t_0(z)}}
=
\frac 1z \left(\frac 12 + \frac{t_0(z)}{\sigma^2}\right)\notag\\
&=
-\eee^{-\ii \theta}\left(\frac \ii 2 \tan \zeta_{\sigma^2/4} (\theta/2) + \frac 12\right)
=
\frac{-\eee^{-\ii \theta}}{1+ \eee^{-2 \ii \zeta_{\sigma^2/4}(\theta/2)}}, \label{eq:S_der}
\end{align}
where we used that $(\partial_1 S)(t_0(z);z) = 0$ since  $t = t_0(z)$ solves the saddle-point equation~\eqref{eq:saddle_point}.
%In the last line, we used~\eqref{eq:saddle_point_formula}.
%\frac{(\partial_t S)(t_0;z)}{S(t_0;z)} = \frac{1}{1+ \eee^{-2 \ii \zeta_{\sigma^2/4}(\theta/2)}}
%=
%\frac 12  + \frac \ii 2 \tan \zeta_{\sigma^2/4} (\theta/2).

\medskip
\noindent
\textit{Contribution of the positive half-axis is negligible.}
We claim that if $r = r(\sigma^2)>0$ is sufficiently small, then uniformly over the disk $|z|\leq r$ it holds that
\begin{equation}\label{eq:int_over_positive_negl_claim}
\int_0^\infty  (1 - z\eee^t)^n \eee^{-\frac n2  \left(\frac{t}{\sigma} + \frac \sigma 2\right)^2}  \dint t
= O(\eee^{ - \frac{1} {16} \sigma^2 n}),
\qquad
\text{ as }
n \to\infty.
\end{equation}
If $0\leq t \leq \sigma^2$, then by choosing a sufficiently small $r>0$ we can achieve that $|1 - z\eee^t|\leq 1 + |z|\eee^t < \eee^{\frac 1 {16}\sigma^2}$ for all $|z|\leq r$ and hence, for all $n\in \N$,
$$
\Big |(1 - z\eee^t)^n \eee^{-\frac n2  \left(\frac{t}{\sigma} + \frac \sigma 2\right)^2}\Big |
\leq
\eee^{\frac{1}{16} \sigma^2 n}  \eee^{- \frac {1}{8}\sigma^2 n}
=
\eee^{-\frac{1}{16} \sigma^2 n}.
$$
It follows that
\begin{equation}\label{eq:int_over_positive_negl_proof_1}
\int_0^{\sigma^2}  (1 - z\eee^t)^n \eee^{-\frac n2  \left(\frac{t}{\sigma} + \frac \sigma 2\right)^2}  \dint t
=O(-\eee^{\frac 1 {16}\sigma^2 n}),
\qquad
\text{ as }
n\to\infty.
\end{equation}
For $t\geq  \sigma^2$ we argue as follows. Since $1\leq \eee^t$ for all $t\geq 0$, we have $|1 - z \eee^t|\leq 1 + |z|\eee^t \leq \eee^t (1+r)$ provided $|z|\leq r$. It follows that
$$
\Big | \int_{\sigma^2}^\infty  (1 - z\eee^t)^n \eee^{-\frac n2  \left(\frac{t}{\sigma} + \frac \sigma 2\right)^2}  \dint t \Big|
\leq
(1+r)^n  \int_{\sigma^2}^\infty \eee^{-\frac n2  \left(\frac{t}{\sigma} - \frac \sigma 2\right)^2}    \dint t
=
(1+r)^n \eee^{-\frac 1 8 \sigma^2 (n-1)} \int_{\sigma^2}^\infty \eee^{-\frac 12  \left(\frac{t}{\sigma} - \frac \sigma 2\right)^2}    \dint t
$$
since $\frac 1 2 (\frac{t}{\sigma} - \frac \sigma 2)^2  \geq \frac 18 \sigma^2$ for $t\geq \sigma^2$.
Note that the integral on the right-hand side converges. Also, if $r>0$ is sufficiently small, then $1+ r \leq \eee^{\frac 1 {16} \sigma^2}$ and it follows that
\begin{equation}\label{eq:int_over_positive_negl_proof_2}
\int_{\sigma^2}^\infty  (1 - z\eee^t)^n \eee^{-\frac n2  \left(\frac{t}{\sigma} + \frac \sigma 2\right)^2}  \dint t
=
O(\eee^{ -\frac 1 {16} \sigma^2 n}),
\qquad
\text{ as }
n\to\infty.
\end{equation}
Combining~\eqref{eq:int_over_positive_negl_proof_1} and~\eqref{eq:int_over_positive_negl_proof_2} yields~\eqref{eq:int_over_positive_negl_claim}, thus completing the proof that the contribution of the positive half-line is negligible.

\medskip
\noindent
\textit{Exact asymptotics of $H_n$.}
It follows from~\eqref{eq:exact_asymptotics} and~\eqref{eq:int_over_positive_negl_claim}, together with the fact that  $S(t_0(z);z)$ converges to $0$ as $z\to 0$, that, for sufficiently small $|z|$, the contribution of the positive half-axis is negligible in the sense that
$$
\int_{-\infty}^\infty \eee^{n S(t;z)}  \dint t
\sim
\sqrt{-\frac{2\pi}{n (\partial_1^2 S)(t_0(z); z)}} \eee^{n S(t_0(z);z)},
\qquad
\text{ as }
n\to\infty.
$$
In view of Lemma~\ref{lem:hubbard_stratonovich}, we can write this as
\begin{equation}\label{eq:hermite_poly_saddle_point_exact_asympt}
\frac{H_n(z;\sigma^2/n)}{(-1)^n} \sim  \sqrt{-\frac{1}{\sigma^2 (\partial_1^2 S)(t_0(z);z)}} \eee^{n S(t_0(z);z)}
\qquad
\text{ as }
n\to\infty.
\end{equation}
This holds pointwise for every  $|z| \leq  r$ provided that $r>0$ is sufficiently small. Although we shall not need this fact, let us mention that after some work it is possible to verify that
\begin{equation}\label{eq:second_der_S}
(\partial_1^2 S)(t_0(z); z)
=
\frac{z^{-1}\eee^{-t_0(z)}}{(1 - z^{-1}\eee^{-t_0(z)})^2} - \frac 1 {\sigma^2}
=
-\frac{1}{4 \cos^2 \zeta_{\sigma^2/4}(\theta/2)} - \frac 1 {\sigma^2}.
\end{equation}

\medskip
\noindent
\textit{Proof of the uniform convergence.}
Now we would like to take the logarithm of both sides of~\eqref{eq:hermite_poly_saddle_point_exact_asympt}. Observe that the left-hand side cannot become zero for $z\in \bD$ by Lemma~\ref{lem:no_zeroes_in_D}. Trivially, the same conclusion holds for the right-hand side. Therefore, applying the function $w\mapsto \log |w|$ to both sides of~\eqref{eq:hermite_poly_saddle_point_exact_asympt} and dividing by $n$ we get
\begin{equation}\label{eq:asympt_H_n_proof_Re}
\lim_{n\to\infty} \frac 1n \log \left|\frac{H_n(z;\sigma^2/n)}{(-1)^n}\right|
=
\Re S(t_0(z);z).
\end{equation}
This holds pointwise in $z$ provided $|z| \leq r$ with $r>0$ sufficiently small. Note that we did not apply the function $w\mapsto \log w$ to avoid difficulties with the choice of the branch.  Now, our aim is to prove that
\begin{equation}\label{eq:asympt_H_n_proof_no_Re}
\lim_{n\to\infty} \frac 1n \log\frac{H_n(z;\sigma^2/n)}{(-1)^n}
=
S(t_0(z);z)
\end{equation}
locally uniformly in $|z| < r$  and with the same convention for the logarithm as in Proposition~\ref{prop:log_asympt_H_n_small_z}.
To this end, we shall use the following lemma which strengthens~\cite[Lemma~3.8]{kabluchko_rep_diff_free_poi}.

%whose proof can be found in~\cite[Lemma~3.8]{kabluchko_rep_diff_free_poi}.
\begin{lemma}\label{lem:differentiate_limit_real_parts}
Let $h_1(z), h_2(z),\ldots$ be a sequence of holomorphic functions defined on some domain $\mathcal D\subset \C$.  If $\Re h_n(z) \to 0$ pointwise on $\cD$  and the sequence  $(\Re h_n(z))_{n\in \N}$ is locally uniformly bounded from above on $\mathcal D$, then $h_n'(z) \to 0$ locally uniformly on $\cD$.
%If, additionally, $h_n(z_0) = 0$ for some $z_0\in \cD$ and all $n\in \N$, then we also have $h_n(z) \to 0$ locally uniformly on $\cD$.
\end{lemma}
%\begin{lemma}\label{lem:differentiate_limit_real_parts}
%Let $h_1(z), h_2(z),\ldots$ be a sequence of holomorphic functions defined on some domain $\mathcal D\subset \C$.  If $\Re h_n(z) %\to 0$ pointwise on $\cD$  and the sequence  $(\Re h_n(z))_{n\in \N}$ is locally uniformly bounded from above on $\mathcal D$, then %$h_n'(z) \to 0$ locally uniformly on $\mathcal D$.
%\end{lemma}
\begin{remark}\label{rem:lem:differentiate_real_parts}
If, additionally, $h_n(y_0) \to 0$ for some $y_0\in \cD$, then, integrating, we conclude that $h_n(z) \to 0$ locally uniformly on $\cD$.
\end{remark}
\begin{proof}[Proof of Lemma~\ref{lem:differentiate_limit_real_parts}.]
Consider some closed disk $B$ contained in $\mathcal D$ and centered at $z_0\in \mathcal D$. We know that $\Re h_n (z)\leq C$, for all $n\in \N$ and $z\in B$. Also, we know that $\Re f_n (z_0)$ converges (and hence is bounded below). By Harnack's inequality applied to the non-negative harmonic functions $C - \Re h_n(z)$ it follows that $C-\Re h_n(z)$ is  uniformly bounded from above on $B$. We conclude that the sequence $(\Re h_n(z))_{n\in \N}$ is locally uniformly bounded, both from above and from below.

Take some $z_0\in \mathcal D$ and let $B_r(z_0)\subset \mathcal D$ be a closed disk of radius $r$ centered at $z_0$ and contained in $\mathcal D$. It is known (see, e.g., the proof of Lemma~3.8 in~\cite{kabluchko_rep_diff_free_poi}) that for all $z$ in the interior of $B_r(z_0)$, we have
\begin{equation}\label{eq:integral_rep_h_n_prime}
h_n'(z) = \frac 1 {\pi \ii} \oint_{|w-z_0| = r} \frac{\Re h_n(w)}{(w-z)^2} \dint w,
\end{equation}
where the integration contour is the boundary of $B_r(z_0)$, oriented counter-clockwise.
%To prove this formula, we may assume that $z=0$ and the contour is some counter-clockwise circle of radius $s$ centered at $0$, by Cauchy's theorem. Since the Taylor series of $h_n(w)$ around $0$ converges uniformly on some neighborhood of this circle, it suffices to verify the representation for individual terms in the Taylor series.
%Thus, we need to check that $\frac 1 {\pi \ii} \oint \Re (a w^\ell) w^{-2} \dint w = a \ind_{\ell = 1}$ for all  $\ell\in \N_0$, $a\in \C$. Substituting $w= s\eee^{\ii \phi}$ with   $\phi\in [-\pi ,+\pi]$, the formula becomes $\frac 1 {\pi} \int_{-\pi}^{+\pi} \Re (a \eee^{\ii \ell \phi}) \eee^{-\ii \phi} \dint \phi = a \ind_{\ell = 1}$, which is elementary to verify.
Having~\eqref{eq:integral_rep_h_n_prime} at our disposal, we claim that $h_n'(z) \to 0$ uniformly over $z\in B_{r/2}(z_0)$.
Indeed, \eqref{eq:integral_rep_h_n_prime} implies
$$
\sup_{z\in B_{r/2}(z_0)} |h_n'(z)| \leq  \frac 1 {\pi} \cdot  \Bigg(\sup_{\substack{|w-z_0| = r\\ z\in B_{r/2}(z_0)}} \frac{1}{|w-z|^2} \Bigg) \cdot  \int_{|w-z_0| = r} |\Re h_n(w)| |\dint w|.
$$
Since the supremum is finite and the integral converges to $0$ by the dominated convergence theorem, the desired conclusion follows.
%%% This is the old proof of uniformity of convergence.
%Recall that $\Re h_n(w)/(w-z)^2 \to 0$ pointwise on $\mathcal D$. Additionally, we know that $|\Re h_n(w)/ (w-z)^2| < C_1$ for all $n\in \N$, $z\in B_{r/2}(z_0)$ and $w\in \partial B_r(z_0)$. While the pointwise convergence $h_n'(z)\to 0$ follows from~\eqref{eq:integral_rep_h_n_prime} together with the dominated convergence theorem, an additional argument is needed to prove uniformity. Let $\eps>0$ be given. By Egorov's theorem, $\Re h_n(w) \to 0$ uniformly provided $w$ stays outside some subset $A_\eps\subset \partial B_r(z_0)$ of one-dimensional Lebesgue measure at most $\eps$. Hence,  $\Re h_n(w)/(w-z)^2 \to 0$ uniformly in  $z\in B_{r/2}(z_0)$ and $w\in \partial B_r(z_0) \backslash A_\eps$. Thus, the integrals in~\eqref{eq:integral_rep_h_n_prime} taken over the complement of $A_\eps$ converge to $0$ uniformly. The integrals over $A_\eps$ can be bounded by $C_1 \eps$, and the claim follows.
\end{proof}

We apply Lemma~\ref{lem:differentiate_limit_real_parts} with $h_n(z) := \frac 1n \log\frac{H_n(z;\sigma^2/n)}{(-1)^n} - S(t_0(z);z)$. These functions are analytic on $\{|z|<r\}$ and we have $h_n(0) = 0$ as well as $\Re h_n(z) \to 0$ pointwise, by~\eqref{eq:asympt_H_n_proof_Re}.  Observe that the functions $(\Re h_n(z))_{n\in \N}$ are locally uniformly bounded from above since by~\eqref{eq:hermite_poly_circ_def1} and the triangle inequality,
$$
\frac 1n \log \left|\frac{H_n(z; \sigma^2/n)}{(-1)^n}\right| \leq \frac 1n \log \left(\sum_{j=0}^n \binom nj |z|^j\right) \leq \log (1+|z|),
\qquad z\in \bC.
$$
Lemma~\ref{lem:differentiate_limit_real_parts} (together with Remark~\ref{rem:lem:differentiate_real_parts}) yields~\eqref{eq:asympt_H_n_proof_no_Re}. Taking into account~\eqref{eq:S_explicit}, this proves Proposition~\ref{prop:log_asympt_H_n_small_z}. Since a locally uniform convergence of analytic functions can be differentiated, we infer that
\begin{align*}
\lim_{n\to\infty} \frac 1n \cdot \frac{H_n'(z;\sigma^2/n)}{H_n(z; \sigma^2/n)}
&=
\frac{\dint}{\dint z} S(t_0(z);z)
=
\frac{-\eee^{-\ii \theta}}{1+ \eee^{-2 \ii \zeta_{\sigma^2/4}(\theta/2)}}
=
-\eee^{-\ii \theta}\left(\frac \ii 2 \tan \zeta_{\sigma^2/4} (\theta/2) + \frac 12\right);
\end{align*}
see~\eqref{eq:S_der}. The proof of Proposition~\ref{prop:log_der_limit_small_r} is complete.
\hfill $\Box$

\section*{Acknowledgements}
The author is grateful to two anonymous referees for several insightful comments and to Octavio Arizmendi, Jorge Garza-Vargas, Jonas Jalowy, Matthias L\"owe and Daniel Perales  for useful discussions and, in particular, for pointing out the works of Mirabelli~\cite{mirabelli_diss} and Hall and Ho~\cite{hall_ho}.
Supported by the German Research Foundation under Germany's Excellence Strategy  EXC 2044 -- 390685587, \textit{Mathematics M\"unster: Dynamics - Geometry - Structure}.

%\addcontentsline{toc}{section}{References}
%:Referenzen

\bibliography{curie_weiss_complex_bib}
\bibliographystyle{plainnat}

\end{document}